\documentclass{amsart} 
\usepackage{color}      
\usepackage{graphicx}

\newtheorem{thm}{Theorem}[section]
\newtheorem{lem}[thm]{Lemma}
\newtheorem{prop}[thm]{Proposition}
\newtheorem{cor}[thm]{Corollary}

\theoremstyle{definition}

\theoremstyle{remark}
\newtheorem{remark}{Remark}[section]
\newcommand{\BUC}{{\rm BUC}}
\newcommand{\tim}{\times}   
\newcommand{\R}{\mathbb R}
\newcommand{\pl}{\partial}
\newcommand{\AC}{{\rm AC}}
\newcommand{\cA}{\mathcal{A}}
\newcommand{\cF}{\mathcal{F}}
\newcommand{\N}{\mathbb N}

\renewcommand{\d}{\,\mathrm{d}}
\newcommand{\cS}{\mathcal{S}}
\newcommand{\USC}{\,\mathrm{USC}}
\newcommand{\LSC}{\,\mathrm{LSC}}
\newcommand{\disp}{\displaystyle}

\newcommand{\gth}{\theta}

\newcommand{\gep}{\varepsilon}
\newcommand{\gd}{\delta}
\newcommand{\gs}{\sigma}

\newcommand{\gb}{\beta}
\newcommand{\fr}{\frac}
\newcommand{\gO}{\Omega}
\newcommand{\hb}{\mbox}
\newcommand{\bye}{\end{document}}
\newcommand{\by}{\end{proof}\end{document}}

\def\dist{\,{\rm dist}\,}
\def\gl{\lambda}
\def\ol{\overline}
\def\gG{\Gamma}
\def\ga{\alpha}
\def\gl{\lambda}

\def\Int{\mathrm{int}}
\def\gS{\Sigma}

\def\go{\omega}

\def\gg{\gamma}
\def\gz{\zeta}
\def\Lip{{\rm Lip}}
\def\mid{\,:\,}
\def\gL{\Lambda}

\numberwithin{equation}{section}

\begin{document}

\def\UC{{\rm UC}}  
  
\def\SP{\mathrm{SP}}  
\def\gL{\Lambda}

\title[Weak KAM theory]{Weak KAM aspects of \\        
convex Hamilton-Jacobi equations with \\     
Neumann type boundary conditions}

\author{Hitoshi Ishii}
\address{Department of Mathematics, Faculty of Education and Integrated Arts and 
Sciences, Waseda University, 1--6--1 Nishi-Waseda, Shinjuku, 
Tokyo, 169-8050 Japan}
\email{hitoshi.ishii@waseda.jp}
\thanks{  
The author was supported in part by KAKENHI \#18204009, \#20340026 
and \#21224001, JSPS}





\date{September 23, 2009}



\begin{abstract}
We study convex Hamilton-Jacobi  equations 
$H(x,Du)=a$ and $u_t+H(x,Du)=a$ in a bounded domain 
$\gO$ of $\R^n$ 
with the Neumann type boundary condition 
$D_\gg u=g$ in the viewpoint of weak KAM theory, 
where $\gg$ is a vector field on the boundary $\pl\gO$ pointing a direction oblique to $\pl\gO$.  
We establish the stability under the formations of infimum and of convex combinations 
of subsolutions of convex HJ equations, 
some comparison and existence results for convex and coercive HJ equations with 
the Neumann type boundary condition as well as existence results for the Skorokhod problem.  
We define the Aubry-Mather set associated with the Neumann type boundary problem 
and establish some properties of the Aubry-Mather set including 
the existence results for the ``calibrated'' extremals for the corresponding 
action functional (or variational problem).       
\end{abstract}

\maketitle

\section{Introduction}           
       
Let $\gO$ be an open connected subset of $\R^n$ with $C^1$ boundary. 
We denote by $\gG$ its boundary $\pl\gO$.  
We consider the Hamilton-Jacobi (HJ for short) equation with the Neumann type 
(or, in other words, oblique)   
boundary condition
\begin{align}  
&H(x,Du(x))=a \ \ \hb{ in }\gO \label{eq:i1}\\
&D_\gg u(x)=g(x) \ \ \hb{ on }\gG. \label{eq:i2}   
\end{align}
Here $a$ is a constant, $H$ is a given continuous function on $\ol\gO\tim\R^n$, called a 
Hamiltonian, $u$ represents the unknown function on $\ol\gO$, 
$Du$ denotes the gradient $(u_{x_1},...,u_{x_n})$, 
$D_\gg u=D_\gg u(x)$ 
denotes the directional derivative $\gg(x)\cdot Du(x)$ at $x$,   
$\gg$ is a continuous vector field: $\gG\to \R^n$, 
and $g$ is a given continuous function on $\gG$.

In addition to the continuity assumption on $H$, $g$, $\gg$, 
we make the following standing assumptions. 
\begin{description}
\setlength{\itemsep}{5pt}
\addtolength{\labelwidth}{20pt}
\item[(A1)] $H$ is a convex Hamiltonian, i.e., for each $x\in\ol\gO$ the function 
$H(x,\cdot)$ is convex on $\R^n$.
\item[(A2)] $H$ is coercive. That is, 
$\disp
\lim_{|p|\to\infty}H(x,p)=\infty.
$ for all $x\in\ol\gO$. 
\item[(A3)] $\gg$ is oblique to $\gG$. That is, for any $x\in\gG$, 
if $\nu(x)$ denotes the outer unit normal vector at $x$, then $\nu(x)\cdot \gg(x)>0$.  
\end{description} 

We consider the initial-value problem with the Neumann type (oblique) boundary condition 
\begin{align}
&u_t(x,t)+H(x,Du(x,t))=a \ \ \hb{ for } (x,t)\in\gO\tim (0,\,T), \label{eq:i3}\\ 
&D_\gg u(x,t)=g(x) \ \ \hb{ for }(x,t)\in \gG\tim (0,\,T), \label{eq:i4}\\
&u(x,0)=u_0(x) \ \ \hb{ for }x\in\ol\gO, \label{eq:i5}
\end{align}
where $0<T\leq\infty$ and $a\in\R$ are given,  
$u$ represents the unknown function on $\ol\gO\tim[0,\,T)$, 
$Du$ denotes the spatial gradient of $u$, $D_\gg u=\gg\cdot Du$,   
and $u_0$ is a given continuous function on $\ol\gO$.   

We call (\ref{eq:i1}) and (\ref{eq:i3}) {\em convex Hamilton-Jacobi equations} if $H$ is a convex 
Hamiltonian.  

The study of weak solutions (i.e., viscosity solutions) 
of problems (\ref{eq:i1}), (\ref{eq:i2}) and 
(\ref{eq:i3})--(\ref{eq:i5}) goes back to Lions  \cite{Lions-Neumann_85}, 
and the theory of existence and uniqueness of viscosity 
solutions of such boundary or initial-boundary value problems including 
the case of second-order elliptic or parabolic equations 
has been well-developed. 
We refer for the developments to \cite{Lions-Neumann_85, LionsTrudinger_86,  BarlesLions_91, DupuisIshii_90, UG, Barles_93} 
and references therein. 
However, if problem (\ref{eq:i1}), (\ref{eq:i2}) has a solution, then it admits clearly multiple solutions 
and therefore the problem is a bit out of the scope of such developments. Indeed, 
problem (\ref{eq:i1}), (\ref{eq:i2}) has a solution only if $a$ is assigned a 
specific value. 

The problem of finding a pair $(a,u)\in\R\tim C(\ol\gO)$ for which $u$ is a 
solution of (\ref{eq:i1}), (\ref{eq:i2}) is called an ergodic problem in terms of 
optimal control or an additive 
eigenvalue problem, and it is also part of weak KAM theory. 
See \cite{LionsPapanicolaouVaradhan} for a classical fundamental work on the ergodic problem for (\ref{eq:i1}) in the periodic setting and also \cite{Fathi-weakKAM, BardiCapuzzo_97}.   

Weak KAM theory concerns the link between the HJ equation (\ref{eq:i1}) 
in a domain $\gO$, with an appropriate boundary condition on its boundary $\pl\gO$, 
and the Lagrangian flow generated by the Lagrangian $L$ given by  
$L(x,\,\xi)=\sup_{p\in\R^n}(\xi\cdot p-H(x,p))$, (or the extremals (minimizers) to the action functional associated with $L$).  
We refer \cite{Fathi_97, E_99, Fathi-weakKAM, Evans-weakKAM_04} for pioneering work and further developments. We refer to \cite{IshiiMitake_07} for some results 
in this direction on HJ equations with the state-constraint boundary condition.  

A typical application of weak KAM theory to the evolution equation 
(\ref{eq:i3}) is in the study of the long-time behavior 
of solutions of (\ref{eq:i3}) with appropriate initial and boundary conditions. 
For these applications we refer to \cite{Fathi_98, Roquejoffre_01, 
DaviniSiconolfi_06, Ishii_08, IchiharaIshii_09, 
Mitake-asymptotic_08, Mitake-large-time_08}.

Our purpose in this paper is to establish some theorems concerning weak KAM theory   
for convex Hamilton-Jacobi equations. Indeed, we define the critical value (or the 
additive eigenvalue) and the Aubry-Mather set associated with (\ref{eq:i1}), 
(\ref{eq:i2}) and 
establish some of basic properties of the Aubry-Mather set, representation formulas 
for solutions of (\ref{eq:i1}), (\ref{eq:i2}) and the existence of extremals (or minimizers) for 
variational formulas of certain types of solutions of (\ref{eq:i1}), (\ref{eq:i2}).  
Our approach is relatively close to that of \cite{FathiSiconolfi_04, 
FathiSiconolfi_05} in view of 
weak KAM theory.  The paper \cite{Serea_07} by O.-S. Serea deals with HJ equations 
on a convex domain with homogeneous 
Neumann condition in view of weak KAM theory. The requirements on the Lagrangian in 
\cite{Serea_07} (see the conditions (7)--(10)) seem very restrictive. On the other hand, 
no regularity on the domain other than the convexity is posed in \cite{Serea_07}.  
In some special cases, the state-constraint problem for (\ref{eq:i1}) is equivalent to   
the Neumann type problem (\ref{eq:i1}), (\ref{eq:i2}), and thus some results in 
\cite{IshiiMitake_07} are related to those obtained here. For this equivalence, we refer 
for instance to \cite{CapuzzoLions_90}.

This paper is organized as follows. In the next section, we establish the 
stability under the formations of infimum and of convex combinations of subsolutions 
of (\ref{eq:i1}), (\ref{eq:i2}) and of (\ref{eq:i3})--(\ref{eq:i5}).  
In Section 3 we establish comparison results for sub and supersolutions of 
(\ref{eq:i1}), (\ref{eq:i2}) and of (\ref{eq:i3})--(\ref{eq:i5}). 
Section 4 is devoted to the Skorokhod problem in $\ol\gO$ with 
reflection direction $\gg$, which is essential to formulate 
variational representations for solutions of (\ref{eq:i1}), (\ref{eq:i2}) and of (\ref{eq:i3})--(\ref{eq:i5}), and we 
establish results concerning existence and stability of solutions of the 
Skorokhod problem. 
In Section 5, we prove the existence of a solution of the initial-boundary value 
problem (\ref{eq:i3})--(\ref{eq:i5}) 
as well as a variational formula for the solution.        
In Section 6, we introduce the critical value and the 
Aubry-Mather set associated with (\ref{eq:i1}), (\ref{eq:i2}), 
study basic properties of the Aubry-Mather set and establish representation      
formulas, based on the Aubry-Mather 
set, for solutions of (\ref{eq:i1}), (\ref{eq:i2}).  
In Section 7 we establish the existence of ``calibrated'' extremals for the variational problem
associated with (\ref{eq:i1}), (\ref{eq:i2}).

{\em Notation: }Let $e_i$, with $i=1,2,...,n$, denote the unit vector of $\R^n$ having 
unity as its $i$th coordinate. We $a\wedge b$ and 
$a\vee b$ for $\min\{a,\,b\}$ and $\max\{a,\,b\}$, respectively. 
For $A\subset\R^n$, $\Lip(A,\,\R^m)$ (resp., $\BUC(A,\,\R^m)$ and $\UC(A,\,\R^m)$) 
denotes the space of 
Lipschitz continuous (resp, bounded uniformly continuous ans uniformly continuous) 
functions on $A$ with values in $\R^m$. For brevity, we may write 
$\Lip(A)$, $\BUC(A)$ and $\UC(A)$ for 
$\Lip(A,\,\R^m)$, $\BUC(A,\,\R^m)$ and $\UC(A,\,\R^m)$, respectively. 
We write $A^c$ to denote the complement of $A$. 
For given function $g$ on $A$ with values in $\R^m$, we write 
$\|g\|_\infty=\sup_{x\in A} |g(x)|$. For an interval $I$, we denote by $\AC(I)$ or $\AC(I,\R^n)$ 
the space of absolutely continuous functions on $I$ with values in $\R^n$. 
For given function $w: A\to\R$ $w^*$ and $w_*$  
denote respectively the upper and lower semicontinuous 
envelopes of $w$ defined on $\ol Q$.  Regarding the definition of (viscosity) solutions, 
we adopt the following convention: for instance, we consider  (\ref{eq:i1}), (\ref{eq:i2}). 
a function $u : \ol\gO\to\R$ is a subsolution (resp., a supersolution) 
provided that $u$ is bounded above (resp., bounded below) and 
whenever $(x,\,\phi)\in\ol\gO\tim C^1(\ol\gO)$ and $u^*-\phi$ 
(resp., $u_*-\phi$) attains a maximum (resp., a minimum)   
at $x$, $H(x,\,D\phi(x))\leq a$ (resp., $\geq a$) if $x\in\gO$ and either 
$H(x,\,D\phi(x))\leq a$ (resp., $\geq a$) or $D_\gg\phi(x)\leq g(x)$ (resp., $\geq g(x)$) 
if $x\in\gG$. 
A bounded function 
$u : \ol\gO\to\R$ is a solution if it is both a subsolution and a supersolution.        
In a more general situation where a candidate of solutions, $u$, is defined on a set which is not necessarily compact, the requirement on $u$ regarding the boundedness 
to be a solution (resp., subsolution or supersolution) is that it is locally 
bounded (resp., locally bounded above or locally bounded below).

\section{Basic propositions on convex HJ equations}

In this section we establish the stability of the operations of infimum and of convex 
combinations subsolutions of 
convex HJ equations. 
We remark that these stability properties, without boundary condition, 
is the main technical observations in the theory 
of lower semicontinuous viscosity solutions due to  Barron-Jensen \cite{BarronJensen_90}.  

To localize problems (\ref{eq:i1}), (\ref{eq:i2}), 
or (\ref{eq:i3})--(\ref{eq:i5}), 
let $U$ be an open subset of $\R^n$ and set 
$\gO_U=U\cap \gO$, $\gG_U=U\cap \gG$ and $\gS:=\gO_U\cup\gG_U=U\cap\ol\gO$.

\subsection{Propositions without the coercivity assumption}
In this subsection we {\em do not assume} the coercivity of $H$. That is, in this subsection 
we assume only (A1) and (A3). Let $f\in C(\gS)$. 
We consider the HJ equation
\begin{equation}\label{eq:b1}
\left\{
\begin{aligned}
&H(x,Du)=f(x) \ \ \hb{ in }\gO_U, \\ 
&D_\gg u(x)=g(x) \ \ \hb{ on }\gG_U, 
\end{aligned}
\right.
\end{equation}
and establish the following theorems. 

\begin{thm}\label{thm:b1} 
Let $\cS\subset \Lip(\gS)$ 
be a nonempty family of subsolutions of {\em (\ref{eq:b1})}. Set 
\[
u(x)=\inf\{v(x)\mid v\in\cS\} \ \ \hb{ for }x\in\gS
\]
and assume that $u\in C(\gS)$. Then $u$ is a 
subsolution of {\em (\ref{eq:b1})}.
\end{thm}

\begin{thm}\label{thm:b2} 
For $k\in\N$ let $f_k\in C(\gS)$ and let $u_k\in \Lip(\gS)$ be a subsolution of {\em (\ref{eq:b1})}, 
with $f_k$ in place of $f$, and $\{\gl_k\}_{k\in\N}$ a sequence of nonnegative numbers such that 
$\sum_{k\in\N}\gl_k=1$. Assume that the sequences $\{u_k\}_{k\in\N}$ and $\{f_k\}_{k\in\N}$ are uniformly 
bounded on compact subsets of $\gS$. 
Set 
\[
u(x)=\sum_{k\in\N}\gl_ku_k(x) \ \ \hb{ and } \ \ f(x)=\sum_{k\in\N}\gl_kf_k(x) 
 \ \ \ \hb{ for }x\in\gS. 
\]
Then $u$ is a 
subsolution of {\em (\ref{eq:b1})}.  
\end{thm}

Before going into the proof of the above two theorems, we give two remarks. 
(i) If $V$ is an open subset of $\R^n$ satisfying $V\cap\ol\gO\subset U$ and 
$u$ is a subsolution (resp., a supersolution) of (\ref{eq:b1}), then 
$u$ is a subsolution (resp., a supersolution) of (\ref{eq:b1}), with $V$ in place of $U$.
(ii) If $U_\ga$ are open subsets of $\R^n$ for $\ga\in\gL$, where $\gL$ is an index set, and 
the inclusion
\[
\ol\gO\subset \bigcup_{\ga\in\gL} U_\ga
\]
holds and $u:\ol\gO\to\R$ is a subsolution of (\ref{eq:b1}), with $U:=U_\ga$, for any $\ga\in\gL$, 
then $u$ is a subsolution (resp., a supersolution) of (\ref{eq:b1}), with $\gO$ and $\gG$ in place of $\gO_U$ 
and $\gG_U$.

In the rest of this subsection we are devoted to proving Theorems \ref{thm:b1} and \ref{thm:b2}. 
It is well-known (see for instance \cite{BarronJensen_90, FathiSiconolfi_04})  
that, if $\gG_U=\emptyset$, 
the assertions of Theorems \ref{thm:b1} and \ref{thm:b2} are valid.  
Thus, in order to prove the above two theorems, because of their local property 
together with the $C^1$ regularity of $\gO$, we
may assume by use of a $C^1$ change of variables that for some constant $r>0$, 
\begin{equation}\label{eq:b2}
U=\Int\, B(0,r), \ \ \gO_U=\{(x',x_n)\in U\mid x_n<0\}, \ \   
\gG_U=\{x=(x',x_n)\in U\mid x_n=0\}.
\end{equation}
Here and later, for $x=(x_1,...,x_n)\in\R^n$, 
we put $x'=(x_1,...,x_{n-1})$ and $x=(x',x_n)$.   

We set $\R^n_+=\R^{n-1}\tim(0,\,\infty)$ and 
define the function $\gz\in C^\infty(\R^n_+\tim\R^n)$ by
$$
\gz(y,z)= 
\fr 12\Big|z-\fr{z\cdot e_n}{y\cdot e_n}y\Big|^2+\fr 12(z\cdot e_n)^2. 
$$
We write $D_{z'}$ for the gradient operator with respect to the variables 
$z'=(z_1,...,z_{n-1})$. For instance, we write 
$D_{z'}\gz=(\gz_{z_1},...,\gz_{z_{n-1}})$.  

\begin{lem}\label{thm:b3}The function 
$\gz\in C^\infty(\R^n_+\tim\R^n)$
has the properties: 
\[\left\{
\begin{aligned}
&\gz(\xi,tz)= t^2\gz(\xi,z) \ \ &&\hb{ for }(\xi,z,t)\in\R^n_+\tim\R^n\tim\R,\\
&\gz(\xi,z)> 0 \ \ &&\hb{ for }(\xi,z)\in\R^N_+\tim(\R^n\setminus\{0\}),\\
&\xi\cdot D_z\gz(\xi,z)=\xi_n z_n \ \ &&\hb{ for }(\xi,z)\in\R^n_+\tim\R^n.   
\end{aligned}
\right.
\]  
\end{lem}

\begin{proof} We observe that 
\begin{align*}
D_z\gz(\xi,z)=&\, z-\fr{z_n}{\xi_n}\xi-\fr{z\cdot \xi }{\xi_n}e_n
+\fr{|\xi|^2z_n}{\xi_n^2}e_n+z_ne_n, \\
\noalign{\noindent and }
\xi \cdot D_z\gz(\xi,z)=&\, \xi_n z_n.
\end{align*}
It is now obvious that the function $\gz$ has all the required properties. 
\end{proof}

We note by the homogeneity of the functions $\gz(\xi,\,\cdot)$ that 
\begin{equation}\label{eq:b+1}
C_0^{-1}|z|^2\leq \gz(\xi,z)\leq C_0|z|^2, \qquad |D_\xi\gz(\xi,z)|  
\leq C_0|z|^2,  \qquad  
|D_z\gz(\xi,z)|\leq C_0|z|
\end{equation}
for all $(\xi,z)\in\R^n_+\tim\R^n$ and for some constant $1<C_0<\infty$.  

By assumption (A3) and (\ref{eq:b2}), we have $\inf_{x\in\gG_U}\gg(x)\cdot e_n>0$.   
We restrict the domain of definition of $\gg$ to $\gG_U$ and then 
extend that of the resulting vector field to $\R^n$ so that 
$\gg\in \BUC(\R^n,\R^n)$ and $\gg_0^{-1} \leq \gg\cdot e_n 
\leq |\gg|\leq \gg_0$ on $\R^n$ for some constant $\gg_0>1$.  Let $\go$ be the modulus of continuity of $\gg$.  

By mollification, we may choose a family of functions 
$\{\gg^\gd\}_{\gd\in(0,\,1)}\subset C^\infty(\R^n,\,\R^n)$ so that 
$|\gg(x)-\gg^\gd(x)|\leq\go(\gd)$, 
$|\gg^\gd(x)-\gg^\gd(y)|\leq\go(|x-y|)$ 
and $|D\gg^\gd(x)|\leq C_1\go(\gd)/\gd$   
for all $x,y\in\R^n$ and $\gd\in(0,\,1)$ and for some constant  
$C_1>1$.  Here $|A|:=\max\{ |A\xi| \,\mid\,\xi\in\R^n,\,|\xi|\leq 1\}$ for $n\tim n$ 
real matrix $A$. We may also assume that  
$\gg_0^{-1}\leq \gg^\gd\cdot e_n \leq|\gg^\gd|\leq \gg_0$ on $\R^n$.

For $\gd\in(0,\,1)$ we set  
$\psi^\gd(x,y)=\gz(\gg^\gd(x),x-y)$ and note that
\begin{align*}
D_x\psi^\gd(x,y)=&\,(D\gg^\gd(x))^{{\rm T}}D_\xi\gz(\gg^\gd(x),x-y)
+D_z\gz(\gg^\gd(x),x-y),\\
D_y\psi^\gd(x,y)=&\,-D_z\gz(\gg^\gd(x),x-y),
\end{align*}
where $A^{\rm T}$ denotes the transposed matrix of the matrix $A$. From these we get
\begin{equation}
|D_x\psi^\gd(x,y)+D_y\psi^\gd(x,y)|=|(D\gg^\gd(x))^{{\rm T}}D_\xi\gz(\gg^\gd(x),x-y)|
\leq \fr{C_0C_1\go(\gd)|x-y|^2}{\gd}. \label{eq:b3}
\end{equation}

Given a bounded function $u$ on 
$\ol\gS$, for $\gd>0$ let $u^\gd\in C(\R^n)$ denote the sup-convolution of $u$ 
with kernel function $\gd^{-1}\psi^\gd$, i.e., 
\[
u^\gd(x)=\sup_{y\in\ol\gS}\left(u(y)-\fr{1}{\gd}\psi^\gd(x,y)\right). 
\]

For $s\in(0,\,r]$ we set
\begin{equation} \label{eq:b3+}\left\{
\begin{aligned}
&\gO_s=\{x=(x_1,...,x_n)\in \Int\, B(0,\,s)\mid x_n<0\},\\
&\gG_s=\{x=(x_1,...,x_n)\in \Int\, B(0,\,s)\mid x_n=0\}.
\end{aligned}
\right.
\end{equation} 
In particular, we have $\gO_U=\gO_r$, $\gG_U=\gG_r$, $\gS=\gO_r\cup\gG_r$ and $\ol\gS=\ol\gO_r$.

\begin{lem}\label{thm:b4}
Let $\mu>0$ and $0<\gep<r$.   Let $u\in \Lip(\gS)$ be 
a viscosity subsolution of {\em(\ref{eq:b1})}, with $f:=0$ and $g:=-\mu$. 
Then there is 
a constant $\gd_0>0$, independent of $u$, such that if $0<\gd<\gd_0$, then 
$v:=u^\gd$ is a viscosity subsolution of 
\begin{equation}
H(x,Dv(x))=\gep \ \ \hb{ in }\gO_{r-\gep}. \label{eq:b5} 
\end{equation}
Moreover, if $0<\gd<\gd_0$, then 
\begin{equation}
D_\gg^+ u^\gd(x)\leq \gep \ \ \hb{ for }x\in\gG_{r-\gep}, \label{eq:b6}
\end{equation}
where 
\[
D_\gg^+v(x):=\limsup_{t\to 0+}\fr{v(x)-v(x-t\gg(x))}{t}. 
\]
\end{lem}

\begin{proof} Let $0<\gd<1$. Let $R>0$ be a Lipschitz constant of $u$. 
We may assume by extending by continuity that $u\in \Lip(\ol\gS)$, so that 
for each $x\in\R^n$ there is a point $y\in\ol\gS$ such that 
\begin{equation}
u^\gd(x)=u(y)-\fr{1}{\gd}\psi^\gd(x,y). \label{eq:b7}
\end{equation}  

Fix $x\in\gO_{r-\gep}\cup\gG_{r-\gep}$ and $y\in\ol\gS$ so that (\ref{eq:b7}) holds. 
We collect here some basic estimates. 
As is standard, we have $u^\gd(x)\geq u(x)$ and 
\[
\fr{1}{\gd}\psi^\gd(x,y)=u(y)-u^\gd(x)\leq u(y)-u(x)\leq R|x-y|. 
\]
Noting by (\ref{eq:b+1}) that $\psi^\gd(x,y)\geq C_0^{-1}|x-y|^2$, we get
\begin{equation}
|x-y|\leq C_2\gd,  \label{eq:b8}
\end{equation}
where $C_2:=C_0R$.   
It follows from (\ref{eq:b3}) that 
\begin{equation}
|D_x\psi^\gd(x,y)+D_y\psi^\gd(x,y)|\leq C_3\go(\gd)\gd, \label{eq:b9}
\end{equation} 
where $C_3:=C_0C_1C_2^2$. By Lemma \ref{thm:b3}, we get
\begin{equation}
\gg^\gd(x)\cdot D_y\psi^\gd(x,y)=-\gg_n^\gd(x)(x_n-y_n). \label{eq:b10} 
\end{equation}  
Also, we get
\begin{align}
|D_y\psi^\gd(x,y)|\leq&\, C_0|x-y|\leq C_4\gd, \label{eq:b11}\\
|D_x\psi^\gd(x,y)|\leq&\, |D_y\psi^\gd(x,y)|+|D_x\psi^\gd(x,y)+D_y\psi^\gd(x,y)|
\leq C_4\gd,\label{eq:b12}
\end{align}
where $C_4:=C_0C_2+C_3\go(1)$.

We now show that $u^\gd$ is a subsolution of (\ref{eq:b5}) if $\gd>0$ is sufficiently small.  
Let $\phi\in C^1(\ol\gO_{r-\gep})$ and $x\in\gO_{r-\gep}$. We assume that 
$u^\gd-\phi$ attains a strict maximum at $x$, and choose a point $y\in\ol\gS=\ol\gO_r$ 
so that (\ref{eq:b7}) holds.  
We choose a constant $\gd_1\in(0,\,1)$ so that $C_2\gd_1<\gep$ and assume in what follows that 
$0<\gd<\gd_1$. By (\ref{eq:b8}), we have $|x-y|<\gep$. Hence, we have 
$\pl\gO_r\setminus \gG_r$. Since $y\in\ol\gO_r$, we have 
two possibilities: $y\in\gO_r$ or $y\in\gG_r$.

Now we treat the case where $y\in\gO_r$.  
Then we have 
\[
D\phi(x)\in D^+u^\gd(x), \ \ 
D\phi(x)+\fr{1}{\gd}D_x\psi^\gd(x,y)=0 \ \ 
\hb{ and } \ \ 
\fr{1}{\gd}D_y\psi^\gd(x,y)\in D^+u(y),
\] 
where $D^+h(x)$ denotes the superdifferential of the function $h$ at $x$.  
Using this last inclusion, we get
$H(y,D_y\psi^\gd(x,y)/\gd)\leq 0$.  
According to (\ref{eq:b11}) and (\ref{eq:b12}), we have $|D_y\psi^\gd(x,y)|/\gd\leq C_4$ and $|D\phi(x)|
=|D_x\psi^\gd(x,y)|/\gd\leq C_4$. 
Let $\go_H$ denote the modulus of continuity of the function $H$ restricted to 
$\ol\gO\tim B(0,\,C_4)$.  Using (\ref{eq:b9}) and (\ref{eq:b8}), we obtain  
\begin{align*}
0\geq&\, H\left(y,\,\fr 1\gd D_y\psi^\gd(x,y)\right)
\geq H(x,D\phi(x))-\go_H(|x-y|)-\go_H(C_3\go(\gd)) \\ 
\geq &\, H(x,\,D\phi(x))-\go_H\left(C_2\gd\right)-\go_H(C_3\go(\gd)).  
\end{align*} 
We choose a $\gd_2>0$ so that 
\[
\go_H\left(C_2\gd_2\right)+\go_H(C_3\go(\gd_2))\leq \gep.
\]
Thus, if $y\in\gO_r$ and $0<\gd <\gd_1\wedge\gd_2$, then we have
\begin{equation}
H(x,\,D\phi(x))\leq \gep. \label{eq:b13}
\end{equation} 

Next, we turn to the case where $y\in\gG_r$. Then we have 
\[
D\phi(x)=-\fr 1\gd D_x\psi^\gd(x,y)\in D^+u^\gd(x) \ \ \hb{ and } \ \ 
\fr 1\gd D_y\psi^\gd(x,y)\in D^+_{\gS}u(y),
\]
where $D_{\gS}^+u(y)$ denotes the set of those $p\in\R^n$ for which 
\[
u(y+\xi)\leq u(y)+p\cdot \xi +o(|\xi|)  \ \ \hb{ as } y+\xi\in\gS \ \hb{ and } \ \xi\to 0. 
\]
By (\ref{eq:b10}), we get 
\[
\gg^\gd(x)\cdot D_y\psi^\gd(x,y)=-\gg_n(x)(x_n-y_n)=-\gg_n(x)x_n>0. 
\] 
Since $|D_y\psi^\gd(x,y)|/\gd\leq C_4$ by (\ref{eq:b11}), 
we get
\begin{align*}
\gg(y)\cdot \fr 1\gd D_y\psi^\gd(x,y)
=&\, \gg^\gd(x)\cdot\fr 1\gd D_y\psi^\gd(x,y)+\left(\gg(y)-\gg^\gd(x)\right)
\cdot\fr 1\gd D_y\psi^\gd(x,y)\cr
>&\,-C_4\left(\go(|x-y|)+\go(\gd)\right)\geq -C_4\left(\go(C_2\gd)+\go(\gd)\right). 
\end{align*}
We select a $\gd_3>0$ so that $
C_4\left(\go(C_2\gd_3)+\go(\gd_3)\right)<\mu$, and 
assume in the following that $0<\gd < \gd_1\wedge \gd_3$. 
Accordingly, we have $\gg(y)\cdot \fr 1\gd D_y\psi^\gd(x,y)>-\mu$. 
Since $u$ is a viscosity subsolution of (\ref{eq:b1}), with $f:=0$ and $g:=-\mu$, 
we get $H\left(y,\,D_y\psi^\gd(x,y)/\gd\right)\leq 0$. Now, as in the previous case, we obtain 
\[
0\geq H\left(x,\,D\phi(x)\right)-\go_H(C_2\gd)-\go_H(C_3\gd). 
\]
Consequently, if $y\in\pl\gO_r$ and 
$0<\gd<\gd_1\wedge \gd_2\wedge \gd_3$, then we have (\ref{eq:b13}). 
Thus we see that if $0<\gd<\gd_1\wedge\gd_2\wedge\gd_3$, 
then $u^\gd$ is a subsolution of (\ref{eq:b5}).

We now prove that (\ref{eq:b6}) is valid if $\gd$ is sufficiently small.  
Let $x\in\gG_{r-\gep}$, 
and ;choose a $y\in\ol\gS$ so that (\ref{eq:b7}) holds. 
Then, for $t>0$ sufficiently small, we have
\[
u^\gd(x)-
u^\gd(x-t\gg(x))\leq -\fr{1}{\gd}\left(\psi^\gd(x,y)-\psi^\gd(x-t\gg(x),y)\right).
\]  
Hence, 
\begin{equation}
D_\gg^+ u^\gd(x)\leq -\gg(x)\cdot \fr 1\gd D_x\psi^\gd(x,y). \label{eq:b14}
\end{equation} 
Using (\ref{eq:b11}), (\ref{eq:b9}) and (\ref{eq:b10}), we compute that
\begin{align}
-\gg(x)\cdot& \fr 1\gd D_x\psi^\gd(x,y)
\leq -\gg^\gd(x)\cdot \fr 1\gd D_x\psi^\gd(x,y)+C_4\go(\gd)\label{eq:b15}\\
\leq&\, \gg^\gd(x)\cdot \fr 1\gd D_y\psi^\gd(x,y) 
+\fr{\gg_0}{\gd}|D_x\psi^\gd(x,y)+D_y\psi^\gd(x,y)|+C_4\go(\gd)\nonumber\\
\leq&\, \gg_0C_3\go(\gd)+C_4\go(\gd).\nonumber 
\end{align}
We select a $\gd_4>0$ so that $(\gg_0C_3+C_4)\go(\gd_4)<\gep$. From (\ref{eq:b14}) 
and (\ref{eq:b15}), we find 
that if $0<\gd<\gd_4$, then (\ref{eq:b6}) holds. 

Finally, setting $\gd_0=\gd_1\wedge\gd_2\wedge\gd_3\wedge\gd_4$, we conclude that if 
$0<\gd<\gd_0$, then $u^\gd$ is a subsolution of (\ref{eq:b5}) and 
satisfies (\ref{eq:b6}). 
\end{proof}

\begin{lem}\label{thm:b5}
Let $\mu>0$. Let $u,\,v\in \Lip(\gS)$ be subsolutions of 
{\em (\ref{eq:b1})}, 
with $f:=0$ and $g:=-\mu$. Then $u\wedge v$ is a subsolution of 
{\em (\ref{eq:b1})}, with 
$f=g=0$. 
\end{lem}

\begin{proof} Fix any $\gep\in(0,r)$. In view of Lemma \ref{thm:b4}, there is a constant 
$\gd_0>0$ such that if $0<\gd<\gd_0$, then $u:=u^\gd,\,v^\gd$ are solutions of 
$H(x,Du)\leq \gep$ in the viscosity sense in $\gO_{r-\gep}$ 
and satisfy $D_\gg^+ u\leq\gep$ on $\gG_{r-\gep}$. As is well-known, 
since $H(x,\cdot)$ is convex, the function $z^\gd:=u^\gd\wedge v^\gd$ is a subsolution of 
$H(x,Dz^\gd)\leq\gep$ in $\gO_{r-\gep}$. 
Also, it is easy to see that $D_\gg^+ z^\gd(x)\leq \gep$ for  
$x\in \gG_{r-\gep}$. It is then easily checked that $z^\gd$ is a subsolution of 
(\ref{eq:b1}), with $\gO_U:=\gO_{r-\gep}$, $\gG_U:=\gG_{r-\gep}$, $f(x):=\gep$ and  
$g(x):=\gep$.    
Sending $\gd\to 0$ and setting $z:=u\wedge v$, we see by the stability of viscosity 
property under uniform convergence that 
$z$ is a viscosity subsolution of (\ref{eq:b1}), 
with $\gO_U:=\gO_{r-\gep}$, $\gG_U:=\gG_{r-\gep}$, $f(x):=\gep$ and $g(x):=\gep$. 
But, since $\gep\in(0,\,r)$ is arbitrary, the function $z$ is a viscosity 
subsolution of (\ref{eq:b1}), with $f:=0$ and $g:=0$.
\end{proof}

Noting that for any $u,v\in C(\gS)$, $0<\gl<1$ and $x\in\gG_U$, 
\[
D_\gg^+ (\gl u+(1-\gl)v)(x)\leq \gl D_\gg^+ u(x)+(1-\gl)D_\gg^+ v(x),  
\] 
we deduce that the argument of the above proof yields also the following lemma.  
  
\begin{lem}\label{thm:b6}
Let $\mu>0$ and $f_1, f_2\in C(\ol\gS)$. For $i=1,2$ let 
$u_i\in C(\ol\gS)$ be a subsolution of 
{\em (\ref{eq:b1}), (\ref{eq:b2})}, 
with $f:=f_i$ and $g:=-\mu$. Let $0<\gl<1$ and set $u=\gl u_1+(1-\gl)u_2$ 
and $f=\gl f_1+(1-\gl)f_2$. Then $u$ is a subsolution of 
{\em (\ref{eq:b1})}, with $g:=0$. 
\end{lem}

\begin{proof}[Proof of Theorem {\em\ref{thm:b1}}]By the continuity of the function $u$, 
we may assume that $\cS$ is a sequence $\{u_k\}_{k\in\N}$. 
Indeed, we can choose a sequence $\{K_m\}_{m\in\N}$ of compact subsets of $\gS$ such that 
$\gS=\bigcup_{m\in\N}K_m$. By a compactness argument, we can choose for each $m\in\N$ 
a sequence $\{v_{m,j}\}_{j\in\N}\subset\cS$ such that 
$u(x)=\inf\{v_{m,j}(x)\mid j\in\N\}$ for $x\in K_m$. Then we have 
$u(x)=\inf\{v_{m,j}(x)\mid j,m\in\N\}$ for $x\in\gS$.  
Relabeling $\{v_{m,j}\}$ appropriately, we find a sequence $\{u_k\}$ which replaces 
$\cS$ in the following argument. 

Next, we fix any $\mu>0$. According to the $C^1$ regularity 
of $\gO$ and the continuity of $g$, we may 
select $\psi_\mu\in C^1(\ol\gO)$ so that 
\[
g(x)+\mu\leq D_\gg\psi_\mu(x)\leq g(x)+2\mu \ \ \hb{ for }x\in\gG.
\]
Set $v_k(x)=u_k(x)-\psi_\mu(x)$ and $v(x)=u(x)-\psi_\mu(x)$ for $x\in\gS$ 
and observe that $w:=v_k$ is a solution of 
\begin{equation}\label{eq:b16}
\left\{
\begin{aligned}
&\widetilde H(x,Dw)\leq 0 \ \ \hb{ in }\gO_U,\\
&D_\gg w(x)\leq-\mu \ \ \hb{ on }\gG_U,
\end{aligned}
\right.
\end{equation}
where $\widetilde H$ is the continuous function on $\ol\gO\tim\R^n$ given by 
$\widetilde H(x,p)=H(x,p+D\psi_\mu(x))-f(x)$. By Lemma \ref{thm:b5}, we see that 
$w_k:=v_1\wedge \cdots \wedge v_k$ is a solution of (\ref{eq:b16}), with $\mu$ 
replaced by $0$.  
Since $w_k(x)\to v(x)$ locally uniformly on 
$\gS$ as $k\to\infty$, by the stability of the viscosity  
property under uniform convergence, we see that $v$ is a solution  
of (\ref{eq:b16}), with $\mu:=0$. This means that $u$ is a subsolution of (\ref{eq:b1}), with $g(x)$ replaced by $g(x)+2\mu$. Since $\mu>0$ is arbitrary, 
we conclude that $u$ is a subsolution of (\ref{eq:b1}).  
\end{proof}

\begin{proof}[Proof of Theorem {\em\ref{thm:b2}}]  
Since the property to be shown is local, by replacing $U$ by a smaller one, we may assume
that the sequences $\{u_k\}$ and $\{f_k\}$ are uniformly bounded on $\gS$.  
Set \[
v_k(x)=\fr{1}{\sum_{j=1}^k\gl_j}\sum_{j=1}^k \gl_j u_j(x) 
\ \ \hb{ and } \ \ F_k(x)=\fr{1}{\sum_{j=1}^k\gl_j}\sum_{j=1}^k \gl_j f_j(x) \ \ \ \hb{ for }x\in\gS.
\]
Assume that $k$ is sufficiently large, so that $\sum_{j=1}^k\gl_j>0$, $v_k\in\Lip(\gS)$ 
and $F_k\in C(\gS)$.   
Moreover, using Lemma \ref{thm:b6} and arguing as in the previous proof 
that $v_k$ is a subsolution of (\ref{eq:b1}), with $f$ replaced by $F_k$. 
In view of the uniform 
boundedness of the sequences $\{u_k\}$ and $\{f_k\}$, we see that $v_k(x)\to u(x)$ 
and $F_k(x)\to f(x)$ uniformly on
$\gS$ as $k\to\infty$. 
By the stability of the viscosity property, we conclude that 
$u$ is a subsolution of (\ref{eq:b1}). \end{proof}

\subsection{Propositions under the coercivity assumption}
In this subsection, we always assume that (A1)--(A3) hold, and reformulate Theorems \ref{thm:b1} and \ref{thm:b2}.

\begin{thm}\label{thm:b7}
Let $\cS\subset C(\gS)$ 
be a nonempty subset of subsolutions of {\em(\ref{eq:b1})}. 
Assume that $\inf\{v(x)\mid v\in\cS\}>-\infty$ 
for some $x\in\gS$. Then the function 
\begin{equation}\label{eq:b+2}
u(x):=\inf\{v(x)\mid v\in\cS\}
\end{equation}
on $\gS$ is a subsolution of {\em(\ref{eq:b1})}.   
\end{thm}

A consequence of the above theorem is stated as follows. If $\cS\subset C(\gS)$ is 
a nonempty subset of solutions of (\ref{eq:b1}) and formula 
(\ref{eq:b+2}) defines a real-valued function $u$, then $u$ 
is a solution of (\ref{eq:b1}). Indeed, as is well-known, 
the supersolultion property is stable under taking infimums, and therefore $u$ is a
supersolution of (\ref{eq:b1}) as well.

\begin{proof}
Because of the local nature of our assertion, 
by replacing $U$ by a smaller one, we may assume
that $f$ is bounded on $\gS$.  
Then, by the coercivity assumption (A2), we can choose a 
constant $C>0$ so that for $(x,p)\in\ol\gO\tim\R^n$, if $H(x,p)\leq f(x)$, then 
$|p|\leq C$. This together with the boundedness and 
$C^1$ regularity of $\gO$ implies that $\cS$ is equi-Lipschitz continuous 
on $\gS$. Consequently, we have $u\in \Lip(\gS)$. Applying Theorem 
\ref{thm:b1}, we find that $u$ is a subsolution of (\ref{eq:b1}).  
\end{proof}

We consider next the evolution equation with the Neumann type 
boundary condition
\begin{equation}\label{eq:b17}
\left\{
\begin{aligned}
&u_t+H(x,Du)=f(x,t) \ \ \hb{ in }\gO_U\tim(0,\,T),\\
&D_\gg u=g(x) \ \ \ \ \hb{ on }\gG_U\tim(0,\,T), 
\end{aligned} 
\right.
\end{equation}
where $f\in C(\gS\tim(0,\,T))$. 

\begin{thm}\label{thm:b8}
Let $\cS\subset C(\gS\tim(0,\,T))$  
be a nonempty subset of subsolutions of {\em(\ref{eq:b17})}. 
Assume that $\cS$ is uniformly bounded on compact subsets of $\gS\tim(0,\,T)$. 
Then the function 
\begin{equation}\label{eq:b+3}
u(x,t):=\inf\{v(x,t)\mid v\in\cS\} 
\end{equation}
on $\gS\tim(0,\,T)$ 
is a subsolution of {\em(\ref{eq:b17})}.   
\end{thm}

A remark parallel to the remark after Theorem \ref{thm:b7} is valid here. 
Indeed, if $\cS\subset C(\gS\tim(0,\,T))$  
is a nonempty subset of solutions of (\ref{eq:b17}) and it is 
uniformly bounded on compact subsets of $\gS\tim(0,\,T)$, then the function 
$u$ given by (\ref{eq:b+3}) is a solution of (\ref{eq:b17}).

\begin{proof} Because the viscosity  
property is local, we may assume, by replacing $U$ and the interval $(0,\,T)$   
by smaller ones and by translation in the $t$-direction if needed, that $\cS$ are uniformly bounded  
on $\gS\tim(0,\,T)$. We may aslo assume that $f\in \BUC(\gS)$. 
Let $C>0$ be a constant such that $|v(x,t)|\leq C$ for $(x,t)\in\gS\tim (0,\,T)$ 
and $v\in\cS$. 

Let $\gep>0$ and introduce the sup-convolution of $v\in\cS$  
with respect to 
the $t$-variable:  
\[
v^\gep(x,t)=\inf_{0<s<T}\Big(v(x,s)-\fr{1}{2\gep}(t-s)^2\Big) \ \ 
\hb{ for }(x,t)\in \gS\tim\R. 
\]   
Setting $\gd=2\sqrt{\gep C}$, we observe that for $(x,t)\in\gS\tim (\gd,\, T-\gd)$, 
\[
v^\gep(x,t)=\max_{|s-t|\leq \gd}\Big(u(x,s)-\fr{1}{2\gep}(t-s)^2\Big),
\]
from which we deduce as usual in viscosity solutions theory that 
$v^\gep$ is a subsolution of 
\begin{equation}\label{eq:b18}
\left\{
\begin{aligned}
&v^\gep_t+H(x,Dv^\gep)=f+\go(\gd) \ \ \hb{ in }\gO_U\tim(\gd,\, T-\gd), \\
&D_\gg v^\gep=g \ \ \ \ \hb{ on } \gG_U\tim (\gd,\,T-\gd), 
\end{aligned}\right.
\end{equation}
where $\go$ is the modulus of continuity of $f$. 

Now, the family of functions $v^\gep(x,\cdot)$, with $x\in\gS$ and $v\in\cS$, 
is equi-Lipschitz continuous on $(\gd,\,T-\gd)$. From this and (\ref{eq:b18}), we see that 
$H(x,Dv^\gep)\leq C_\gep$ in the viscosity sense in $\gO_U\tim(\gd,\,T-\gd)$ 
for all $v\in\cS$ and for some constant $C_\gep>0$. 
Observe then that for $(x,t)\in\gS\tim\R$, 
\[
u^\gep(x,t):=\inf_{0<s<T}\Big(u(x,s)-\fr{1}{2\gep}(t-s)^2\Big) 
=\inf\{v^\gep(x,t)\mid v\in\cS\}. 
\]
We apply Theorem \ref{thm:b1}, to see that $u^\gep$ is a subsolution 
of (\ref{eq:b18}).  Indeed, in order to apply Theorem \ref{thm:b1}, we set 
$\widetilde\gO=\gO\tim(0,\,T)$, $\widetilde U=U\tim(0,\,T)$, 
$\widetilde H(x,t,p,q)=H(x,p)+q$ and $\tilde\gg(x,t)=(\gg(x),\,0)$, and 
regard problem (\ref{eq:b17}) as problem 
(\ref{eq:b1}),    
with $\widetilde \gO$, 
$\tilde U$, $\widetilde H$ and $\tilde\gg$ 
in place of $\gO$, $U$, $H$ and $\gg$, respectively. 

Next, we observe that for $(x,t)\in\gS\tim(0,\,T)$, the family  
$\{u^\gep(x,t)\}$ converges monotonically to $u(x,t)$ as $\gep\to 0$, which implies,
together with the continuity of $u^\gep$,  
that $u(x,t)$ is identical to the upper relaxed limit of $u^\gep(x,t)$ as $\gep\to 0$. 
Because of the stability of the subsolution property under such a limiting process, 
we see that $u$ is a subsolution of (\ref{eq:b17}).    
\end{proof}

\begin{thm}\label{thm:b10} 
For $k\in\N$ let $f_k\in C(\gS\tim(0,\,T))$ and    
$u_k\in \USC(\gS\tim(0,\,T))$ be a subsolution of {\em (\ref{eq:b17})},  
with $f_k$ in place of $f$. Let  
$\{\gl_k\}_{k\in\N}$ be a sequence of nonnegative numbers such that 
$\sum_{k\in\N}\gl_k=1$. Assume that the sequences $\{u_k\}_{k\in\N}$ 
and $\{f_k\}_{k\in\N}$ are uniformly 
bounded on compact subsets of $\gS\tim(0,\,T)$. 
Set 
\[
u(x,t)=\sum_{k\in\N}\gl_ku_k(x,t) \ \ \hb{ and } \ \ f(x,t)=\sum_{k\in\N}\gl_kf_k(x,t) 
 \ \ \ \hb{ for }x\in\gS\tim(0,\,T). 
\]
Then $u$ is a 
subsolution of {\em (\ref{eq:b17})}.  
\end{thm}

\begin{proof} 
Arguing as in the proof of Theorem \ref{thm:b8}, with use of Theorem \ref{thm:b1} 
instead of Theorem \ref{thm:b2}, we conclude that Theorem 
\ref{thm:b10} is valid.  
\end{proof}

\section{Comparison results}

The comparison results presented in this section are more or less well-known 
(see for instance \cite{Lions-Neumann_85, BarlesLions_91, DupuisIshii_90}). 
A only new feature of our results 
may be in the point that they are formulated in a localized fashion.

\begin{thm}\label{thm:c1}
Let $f_1,f_2\in C(\gS)$ and let $u\in\USC(\ol\gS)$ {\em(}resp., $v\in\LSC(\ol\gS)${\em)} 
be a subsolution {\em(}resp., a supersolution{\em)} of {\em(\ref{eq:b1})},  
with $f$ replaced by $f_1$ {\em(}resp., $f_2${\em)}. 
Assume that $f_1(x)<f_2(x)$ for $x\in\gS$. Then 
\[
\sup_{\gS}(u-v)\leq \sup_{\pl U\cap\ol\gO}(u-v).
\]  
\end{thm} 

We remark that if $\pl U\cap\ol\gO=\emptyset$ in the above theorem, 
then the right side of the above 
inequality equals $-\infty$ by definition. In particular, if $\ol\gO\subset U$ in the above theorem, then the theorem asserts that $\sup_{\gO}(u-v)=-\infty$.  

\begin{cor}\label{thm:c2} If $a<b$ and problem {\em(\ref{eq:i1}), (\ref{eq:i2})} has a 
subsolution, 
then problem {\em(\ref{eq:i1}), (\ref{eq:i2})}, with $b$ in place of $a$, does not 
have a supersolution.  In particular, if problem {\em(\ref{eq:i1}), (\ref{eq:i2})} 
has a solution for some $a\in\R$, then problem {\em(\ref{eq:i1}), (\ref{eq:i2})}, 
with $a$ replaced by $b\not=a$, has no solution.  
\end{cor}

\begin{proof} Let $a<b$, and assume that there are a subsolution $u$ of (\ref{eq:i1}), 
(\ref{eq:i2}) and a supersolution of (\ref{eq:i1}), (\ref{eq:i2}), with $b$ in place of $a$. 
Note that, for any $c\in\R$, the function $u+c$ is also a subsolution of (\ref{eq:i1}), 
(\ref{eq:i2}).  
By Theorem \ref{thm:c1}, we have $u^*+c\leq v_*$ on $\ol\gO$ for $c\in\R$, which is
a contradiction. This proves our claim.   
\end{proof}

\begin{lem}\label{thm:c3}Assume that $f$ is bounded on $\gS$. Then 
there is a constant $C>0$, depending only $H$, $f$ 
and $\gO$, such that for any subsolution $u\in\USC(\gS)$ of {\em(\ref{eq:b1})} 
and $x,y\in\gS$,  
$|u(x)-u(y)|\leq C|x-y|$.  
\end{lem}

\begin{proof} Let $u\in\USC(\gS)$ be a subsolution of (\ref{eq:b1}). 
By the coercivity assumption (A2) and the boundedness of $f$, 
there is a constant $C_0>0$ 
such that for $(x,p)\in\gO_U$, if $|p|\geq C_0$, then $H(x,p)\geq f(x)+1$.   
It follows from (\ref{eq:b1}) 
that $u$ is a subsolution of $|Du|\leq C_0$ in $\gO_U$, which implies 
together with the $C^1$ regularity of $\gO$ that $u$ is 
Lipschitz continuous on $\gO_U$ with a Lipschitz constant $C>0$ 
depending only on $C_0$ and $\gO$. 

We next show that $u\in C(\gS)$, which guarantees that $u$ is Lipschitz 
continuous on $\gS$ with the same Lipschitz constant $C$.   
To this end, we need only to show that for any fixed $z\in\gG_U$, 
$u$ is continuous at $z$. By translation, we may assume that $z=0$. 
By rotation and localization, we may furthermore assume that $U$, $\gO_U$ 
and $\gG_U$ are given by (\ref{eq:b2}). Since $u\in\USC(\gS)$ and 
$u\in\Lip(\gO_U)$, it is enough to show that 
\begin{equation}\label{eq:c1}
u(0)\leq \sup_{\gO_s}u \ \ \hb{ for }s\in(0,\,r).
\end{equation}
Here and later we use the notation $\gO_s$ and $\gG_s$ as defined in (\ref{eq:b3+}). 

We may assume by replacing $r>0$ by a smaller one that 
$\gg_0:=\inf_{x\in\gG_r}\gg(x)\cdot e_n>0$. (Recall that $e_n$ denotes the unit vector $(0,...,0,1)\in\R^n$.) We select a closed convex cone $K$ with vertex at the origin 
so that $K\setminus\{0\}\subset -\R_+^n$ and $-\gg(x)+B(0,\gd)\in K$ 
for all $x\in\gG_r$ and for some $\gd>0$. 
We denote by $N_K$ the normal cone to $K$ at the origin. That is, we set 
$N_K=\{\xi\in\R^n\mid \xi\cdot p\leq 0 \ \hb{ for }p\in K\}$. 
It follows that $\xi\cdot(-\gg(x))\leq -\gd|\xi|$ for all $\xi\in N_K$ 
and $x\in\gG_r$. Let $d_K$ denote the distance function from the set $K$, i.e., 
$d_K(x)=\dist(x,K)$. As is well-known, the function $d_K$ is convex on $\R^n$, 
$d_K\in C(\R^n)\cap C^1(\R^n\setminus K)$, $d_K(x)\geq 0$ for $x\in\R^n$ 
and $Dd_K(x)\in N_K\cap \pl B(0,1)$ 
for $x\in\R^n\setminus K$.

Fix any $s\in(0,\,r)$ and 
set $\rho=\dist(K,\,\pl B(0,\,s)\cap\{x_n=0\})$. Here and later we use the notation: 
$\{x_n=0\}:=\{(x',x_n)\in\R^n\mid x_n=0\}$ and similarly $\{x_n<0\}:=
\{(x',x_n)\in\R^n\mid x_n<0\}$.  
Note that $0<\rho\leq s$ and fix any $\gep\in (0,\,\rho)$. We may assume by replacing $r>0$ by a 
smaller one that $u$ is bounded above on $\ol\gO_r$. We choose a constant $C_1>0$ 
so that $\sup_{\ol\gO_r}u\leq C_1$, $\sup_{\gO_r}|u|\leq C_1$ and $\sup_{\gG_r}g\leq C_1$. 
We select a function $\gz_\gep\in C^1(\R)$ so that $\gz_\gep'(r)\geq 1$ for $r\in\R$, 
$\gz_\gep(r)=r$ for $r\leq \gep$ and $\gz_\gep(\rho)\geq 2C_1$.   
We set $A=\max\{1,\,C_0,\,(C_1+1)/\gd\}$, and 
define 
the function $v\in C(\R^n)$ by 
\[
v(x):=A\gz_\gep(d_K(x+\gep e_n))+\sup_{\gO_s}u=A\gz_\gep(\dist(x,\,K-\gep e_n)+\sup_{\gO_s}u.  
\]
Let $V=\gO_s\setminus (K-\gep e_n)$. 
We intend to show that $u\leq v$ on the set $\ol V$. To do this, we suppose by contradiction 
that $\max_{\ol V}(u-v)>0$. 
Note that 
\[\ol V\subset \ol\gO_s=\left(\ol\gO_s\cap \{x_n<0 \}\right) 
\cup \left(
\pl B(0,s)\cap\{x_n=0\}\right)\cup \gG_s.\]
\begin{center}
\includegraphics[height=5cm]{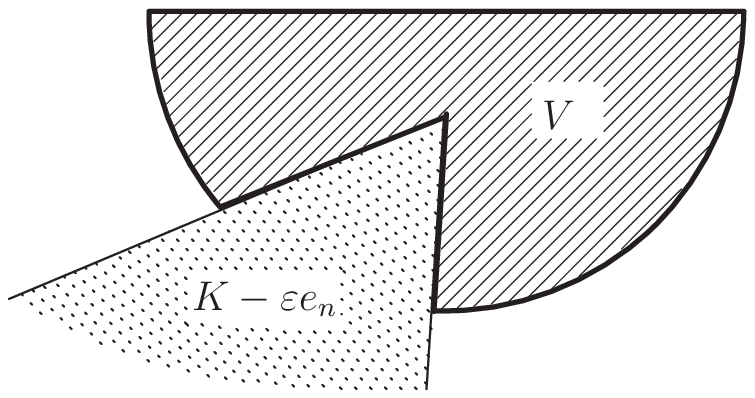}         
\end{center} \vspace{-20pt}
Since $u\in C(\gO_r)$, it is clear 
that $u\leq \sup_{\gO_s}u\leq v$ on $\ol\gO_s\cap\{x_n<0\}$.  
For any $x\in\pl B(0,s)\cap\{x_n=0\}$, we have 
$\dist(x,\,K-\gep e_n)\geq \dist(x,\,K)\geq \rho >s$ and hence 
\[v(x)\geq 
\gz_\gep(d_K(x+\gep e_n))-C_1\geq \gz_\gep(\rho)-C_1\geq C_1\geq u(x).\] 
Consequently, we have $u(x)\leq v(x)$ for $\ol V\cap \gG_s$ and therefore 
there is a point $y\in \gG_s$ such that $(u-v)(y)=\max_{\ol V}(u-v)$. 
Since $u$ is a subsolution of (\ref{eq:b1}), with $V$ in place of $U$, 
we have either $H(y,\,Dv(y))\leq f(y)$ or $D_\gg v(y)\leq g(y)$. 
Since $y\in\gG_s$ and $\gG_s\cap (K-\gep e_n)=\emptyset$, 
we have 
\[
Dv(y)=A\gz_\gep'(d_K(y+\gep e_n))Dd_K(y+\gep e_n).  
\]
Hence, we get $|Dv(y)|\geq A\geq C_0$ 
and, by the choice of $C_0$, $H(y,\,Dv(y))>f(y)$. 
Also, we get
\[
D_\gg v(y)=A \gz_\gep'(d_K(y+\gep e_n))\gg(y)\cdot Dd_K(y+\gep e_n)
\geq A \gd\geq C_1+1 >g(y).
\]
We are in a contradiction, and thus we conclude that (\ref{eq:c1}) holds.  
\end{proof}

\begin{proof}[Proof of Theorem {\em\ref{thm:c1}}] We first deal with the case 
where $\ol\gO\cap\pl U\not=\emptyset$. 
We suppose by contradiction that 
\begin{equation}\label{eq:c2}
\max_{\ol\gS}(u-v)>\max_{\pl U\cap \ol\gO}(u-v). 
\end{equation} 

By replacing $U$ by a smaller one (for instance, the set $\{x\in U\mid 
\dist(x,\pl U)>\gep\}$ with sufficiently small $\gep>0$) if needed, we may assume that 
$f_1,\,f_2$ are continuous on $\ol\gS$ and 
$\sup_{\ol\gS}(f_1-f_2)<0$. We note by Lemma \ref{thm:c3} that 
the function $u$ is Lipschitz continuous on $\gS$.   

We now intend to replace $H$ by a uniformly continuous Hamiltonian, which is not 
coercive nor convex any more.  For this, we define the function $\widetilde H
\in\UC(\gS\tim\R^n)$ by 
\[
\widetilde H(x,p)=\min\{H(x,p)-f_1(x),\, 1\}. 
\]  
Set $\tilde f_1(x)=0$ and $\tilde f_2(x)=\min\{f_2(x)-f_1(x),\,1\}$ for $x\in\gS$.  
Now, the function $u$ (resp., $v$) is  
a subsolution (resp., a supersolution) of (\ref{eq:b1}), with $\widetilde H$ and 
$\tilde f_1$ (resp., $\tilde f_2$) in place of $H$ and $f$. 
Thus, replacing $H$, $f_1$ and $f_2$ by $\widetilde H$, $\tilde f_1$ and 
$\tilde f_2$, respectively, we may assume in what follows that $H\in\UC(\gS\tim\R^n)$.

We select a function $\psi\in C^1(\ol\gO)$ so that $D_\gg\psi(x)>0$ for all 
$x\in\gG$. Let $\gd>0$ and set 
\[
u_\gd(x)=u(x)-\gd\psi(x) \ \hb{ and } \ v_\gd(x)=v(x)+\gd \psi(x) 
\ \ \hb{ for }\, x\in\ol\gS.
\]
In view of the uniform continuity of $H$, selecting $\gd>0$ small enough, 
replacing $f_1,\,f_2$ by a new ones if necessary, we may assume that 
$u_\gd$ (resp., $v_\gd$) is a subsolution (resp., a supersolution) of 
(\ref{eq:b1}), with $g$ and $f$ replaced respectively 
by $g-\gep$ (resp., $g+\gep$), where $\gep$ 
is a positive constant and by $f1$ (resp., $f_2$).  
We may also assume that (\ref{eq:c2}) holds with 
$u_\gd$ and $v_\gd$ in place of $u$ and $v$, respectively. 
Henceforth we replace $u$ and $v$ by 
$u_\gd$ and $v_\gd$ in our notation, respectively.  

If $\sup_{\gG_U}(u-v)<\max_{\ol\gS}(u-v)$, then we have $\max_{\pl\gS}(u-v) 
<\max_{\ol\gS}(u-v)$ and get a contradiction by 
arguing as in the standard proof (in the case of the Dirichlet boundary condition) 
of comparison results where the Lipschitz 
continuity of $u$ is available.  

Thus we assume henceforth that $\sup_{\gG_U}(u-v)=\max_{\ol\gS}(u-v)$. 
Then the function $u-v$ attains a maximum at a point $z\in\gG_U$.  
By replacing $U$ by an open ball $\Int B(z,r)$, with $r>0$ sufficiently small, 
and by translation,      
we may assume that $z=0$, $\gO_U=\gO_r$ and $\gG_U=\gG_r$, where $\gO_r$ and $\gG_r$ 
are the sets given by (\ref{eq:b3+}).  
We set $\tilde\gg=\gg(0)/|\gg(0)|^2$,  
\[
\tilde u(x)=u(x)-g(0)\tilde\gg\cdot x-|x|^2 \ 
\hb{ and } \ \tilde v(x)=v(x)-g(0)\tilde\gg\cdot x \ \  \hb{ for }\,\ol\gS.\]
Note that 
$\tilde u-\tilde v$ attains a strict maximum at the origin and that 
$w:=\tilde u$ is a solution of
\[\left\{
\begin{aligned}
&H(x,Dw(x)+g(0)\tilde\gg+2x)\leq f_1(x) \ \ \ \hb{ in }\gO_r,\\
&D_\gg Dw(x)\leq g(x)-g(0)\gg(x)\cdot\tilde\gg-2\gg(x)\cdot x-\gep \ \ \hb{ on }\gG_r,
\end{aligned}
\right.
\]  
and $w:=\tilde v$ is a solution of
\[\left\{
\begin{aligned}
&H(x,Dw(x)+g(0)\tilde\gg)\geq f_2(x) \ \ \hb{ in }\gO_r,\\
&D_\gg Dw(x)\geq g(x)-g(0)\gg(x)\cdot\tilde\gg+\gep \ \ \hb{ on }\gG_r,
\end{aligned}
\right.
\]
Replacing $r>0$ 
by a smaller positive number, we may assume that $w:=\tilde u$ is a solution of 
\[\left\{
\begin{aligned}
&H(x,Dw(x)+g(0)\tilde\gg)\leq f_1(x)+\gep \ \ \ \hb{ in }\gO_r,\\
&D_\gg Dw(x)\leq -\fr \gep 2 \ \ \hb{ on }\gG_r,
\end{aligned}
\right.
\]  
and $w:=\tilde v$ is a solution of 
\[\left\{
\begin{aligned}
&H(x,Dw(x)+g(0)\tilde\gg)\geq f_2(x) \ \ \hb{ in }\gO_r,\\
&D_\gg Dw(x)\geq \fr \gep 2 \ \ \hb{ on }\gG_r,
\end{aligned}
\right.
\]
Reselecting $\gep>0$ small enough if necessary, 
we may assume that $\max_{\ol\gO_r}(f_1+\gep-f_2)<0$.  
In the argument which follows, 
we write $u$, $v$, $f_1$ and $H$ for the functions $\tilde u$, $\tilde v$, 
$f_1+\gep$ and $H(x,p+g(0)\tilde\gg)$, respectively. 

Let $\gz\in C^\infty(\R_+^n\tim\R^n)$ be the function from Lemma \ref{thm:b3}. 
Set $\phi(x,y)=\gz(\gg(0),x-y)$. For $\ga>1$ we consider the function 
$
\Phi(x,y):=u(x)-v(y)-\ga\phi(x,y)
$
on $\ol\gS\tim\ol\gS$. Let $(x_\ga,y_\ga)\in\ol\gS^2$ be a maximum 
point of $\Phi$. Since $u-v$ attains a strict maximum at the origin, 
we deduce easily that $x_\ga,\,y_\ga\to 0$ as $\ga\to\infty$. 
Let $C_1>0$ be the Lipschitz constant of the function $u$. Then, 
since $\Phi(y_\ga,y_\ga)\leq\Phi(x_\ga,y_\ga)$, we find that  
$\ga\phi(x_\ga,u_\ga)\leq C_1|x_\ga-y_\ga|$, from which we get 
$\ga |x_\ga-y_\ga|\leq C_2$, where $C_2>0$ is a constant independent of $\ga$. 
If $x_\ga,\,y_\ga\in\gO_r$, then we have 
\[
H(x_\ga, D_x\phi(x_\ga,y_\ga))\leq f_1(x_\ga) \ \hb{ and } \  
H(y_\ga, -D_y\phi(x_\ga,y_\ga))\geq f_2(y_\ga). 
\]
Here, noting that $D_x\phi(x,y)+D_y\phi(x,y)=0$, we find that
\begin{equation}\label{eq:c3}
H(x_\ga, D_x\phi(x_\ga,y_\ga))\leq f_1(x_\ga) \ \hb{ and } \  
H(y_\ga, D_x\phi(x_\ga,y_\ga))\geq f_2(y_\ga). 
\end{equation} 
Assume instead that $x_\ga\in\gG_r$. By the viscosity property of $u$, we have 
either 
\[
H(x_\ga, D_x\phi(x_\ga,y_\ga))\leq f_1(x_\ga) \ \hb{ or } \   
\gg(x_\ga)\cdot D_x\phi(x_\ga,y_\ga)\leq-\fr \gep 2.
\]
Compute that
\[
\gg(x_\ga)\cdot D_x\phi(x_\ga,y_\ga)
=\gg(x_\ga)\cdot D_z\gz(\gg(0),x_\ga-y_\ga)
\geq \gg_n(0)\cdot (-y_{\ga n})-C_2C_3\go_\gg(|x_\ga|), 
\]
where $C_3>0$ is a constant, independent of $\ga$, 
such that $|D_z\gz(\gg(0),z)|\leq C_3|z|$ for 
$z\in\R_+^n\tim\R^n$, $\go_\gg$ is the modulus of 
continuity of $\gg$ on $\gG$ and $y_{\ga n}:=e_n\cdot y$.  
Accordingly, if $\ga$ is large enough, then we have
\[
\gg(x_\ga)\cdot D_x\phi(x_\ga,y_\ga)>-\fr \gep 2. 
\]
Thus, we have $H(x_\ga,\,D_x\phi(x_\ga,y_\ga))\leq f_1(x_\ga)$ 
if $\ga$ is large enough. 
Similarly, in the case where $y_\ga\in\gG_r$, we have 
$H(y_\ga, D_x\phi(x_\ga,y_\ga))\geq f_2(y_\ga)$ 
if $\ga$ is large enough. 
Now, assuming $\ga$ is large enough, we always have (\ref{eq:c3}), from which get 
a contradiction, $f_1(0)\geq f_2(0)$, by taking the limit as $\ga\to\infty$.   

We next turn to the case where $\pl U\cap\ol\gO=\emptyset$.  
We have
\[
\gO=(\gO\cap U)\cup (\gO\cap U^c)=(\gO\cap U) \cup\left(\gO\cap \Int(U^c)\right).
\]
Since $\gO$ is connected and $\gO\cap U=\gS\not=\emptyset$, 
we see that $\gO\cap \Int(U^c)=\emptyset$ and 
$\ol\gO\subset U$. 
We thus need to show that 
\[
\sup_{\ol\gO}(u-v)=-\infty. 
\]
Indeed, if $\max_{\ol\gO}(u-v)\in\R$, then the argument in the previous case 
yields a contradiction. The proof is now complete. 
\end{proof}

\begin{thm}\label{thm:c4}  
Let $u\in\USC(\ol\gS\tim[0,\,T))$ and $v\in\LSC(\ol\gS\tim[0,\,T))$  
be respectively a subsolution and a supersolution 
of {\em(\ref{eq:b17})}.  
Assume that  
$u\leq v$ on $\ol\gS \tim\{0\}\,\cup\,(\pl U\cap\ol\gO)\tim(0,\,T)$. 
Then $u\leq v$ in $\ol\gS\tim[0,\,T)$.    
\end{thm}

\begin{lem}\label{thm:c5}
Assume that $f\in C(\gS\tim(0,\,T))$ is bounded on $\gS\tim(0,\,T)$. 
Then for any $R>0$ there is a constant $C_R>0$, depending only on $R$, 
$H$, $f$ and $\gO$, for which 
if $u\in\USC(\gS\tim(0,\,T))$ is a subsolution of {\em(\ref{eq:b17})}
and if the family $\{u(x,\cdot)\mid x\in\gS\}$, is equi-Lipschitz 
continuous on $(0,\,T)$ with Lipschitz constant $R$,   
then the function $u$ is Lipschitz continuous on $\gS\tim(0,\,T)$ with 
Lipschitz constant $C_R$.      
\end{lem}

\begin{proof} Fix any $R>0$. As in the proof of Lemma \ref{thm:c3}, there is a constant 
$M_R>0$, 
depending only on $R$, $H$ and $f$, such that 
for $(x,\,p)\in \gS\tim\R^n$, if $H(x,p)\leq f(x)+R$, then 
$|p|\leq M_R$. Let $u\in\USC(\gS\tim(0,\,T))$ be  
a subsolution of (\ref{eq:b17}), and assume that 
the family $\{u(x,\cdot)\mid x\in\gS\}$ is 
equi-Lipschitz continuous on $(0,\,T)$ with Lipschitz constant $R$. 
Then, it is easily seen that for each $t\in(0,\,T)$, 
the function $u(\cdot,t)$ is a subsolution of (\ref{eq:b1}), with 
$H(x,p)$ and $f(x)$ replaced by $|p|$ and $C_0$, respectively. 
By Lemma \ref{thm:c3}, there is a constant $C_R\geq R$, depending only on 
$M_R$ and $\gO$, such that the family $\{u(\cdot,t)\mid 0<t<T\}$ is equi-Lipschitz 
continuous on $\gS$, with Lipschitz constant $C_R$. Then we have 
$|u(x,t)-u(y,s)|\leq C_R(|x-y|+|t-s|)$ for all $(x,t),\, (y,s)\in\gS\tim(0,\,T)$ and finish 
the proof.      
\end{proof}

\begin{proof}[Proof of Theorem {\em\ref{thm:c4}}]
We follow the line of the proof of Theorem \ref{thm:c1}. 
For $S<T$ we write   
\[
\pl'_p (\gS\tim (S,\,T))=\ol\gS\tim\{S\}\,\cup\, \left(\pl U\cap\ol\gO\right) \tim (S,\,T).    
\]
It is enough to show that 
\begin{equation}\label{eq:c4}
\sup_{Q_T}(u-v)\leq \sup_{\pl'_p Q_T}(u-v), 
\end{equation}
where $Q_T=\gS\tim (0,T)$.  

To prove (\ref{eq:c4}), we suppose, on the contrary, that 
\begin{equation}\label{eq:c5}
\sup_{Q_T}(u-v)>\sup_{\pl'_p Q_T}(u-v). 
\end{equation}
Let $\gd>0$ and set 
\[
\tilde u(x,t)=u(x,t)-\fr{\gd}{T-t} \ \ \hb{ for }(x,t)\in Q_T. 
\]
By replacing $u$ by $\tilde u$, we may assume that $u$ is a subsolution 
of (\ref{eq:b17}) with $f(x)$ replaced by $f(x)-\gep$, where $\gep>0$ is a constant,   
and that 
\[
\lim_{t\to T-}\sup_{x\in\ol\gS}(u-v)(x,t)=-\infty.  
\]   
By taking the sup-convolution of $u$ in the $t$-variable, 
replacing $U$ and the interval $(0,\,T)$ by smaller (in the sense of inclusion) ones, 
and translating the smaller interval, 
we may assume that $f$ is uniformly continuous on $Q_T$ and the family 
$\{u(x,\cdot)\mid x\in\ol\gS\}$ is equi-Lipschitz continuous 
on $(0,\,T)$. According to Lemma \ref{thm:c5}, 
the function $u$ is Lipschitz continuous on $Q_T$. 
Next, we may replace $H$ by a uniformly continuous function on $\gS\tim\R^n$. 
By perturbing $u$ (resp., $v$) as in the proof of Theorem \ref{thm:c1} and 
replacing $\gep>0$ by a smaller positive number, 
we may assume that $u$ (resp., $v$) is a subsolution (resp., a supersolution) 
of (\ref{eq:b17}), with $f(x,t)$ and $g(x)$ replaced by $f(x,t)-\gep$ and $-\gep$ 
(resp., $f(x,t)$ and $\gep$). Moreover, we may assume that $u-v$ attains a
strict maximum at a point $(z,\tau)\in\gG_U\tim(0,\,T)$.  
Furthermore, we may assume that $z=0$, $U=\Int B(0,r)$, $\gO_U=\gO_r$ and 
$\gG_U=\gG_r$, where $r>0$ and $\gO_r,\, \gG_r$ are the sets given by (\ref{eq:b3+}). 

Now we consider the function 
\[
\Phi(x,t,y,s)=u(x,t)-v(y,s)-\ga\phi(x,y)-\ga(t-s)^2
\]
on the set $\ol Q_T \tim \ol Q_T$, where 
$\ga>1$ is a constant and $\phi$ is the function used in the proof of Theorem \ref{thm:c1}. 
Let $(x_\ga,y_\ga)\in \ol Q_T\tim \ol Q_T$ be a maximum point of $\Phi$. 
Arguing as in the proof of Theorem \ref{thm:c1}, we see that if $\ga$ is sufficiently 
large, then we always have 
\begin{equation}\label{eq:c6}
\left\{
\begin{aligned}
&2\ga(t_\ga-s_\ga)+H(x_\ga,\,\ga D_x\phi(x_\ga,\,y_\ga))\leq f(x_\ga)-\gep, \\
&2\ga(t_\ga-s_\ga)+H(y_\ga,\,\ga D_x\phi(x_\ga,\,y_\ga))\geq f(y_\ga).
\end{aligned}\right.
\end{equation}
Also, using the Lipschitz continuity of $u$, we find that for some constant $C>0$, 
independent of $\ga$, 
\[
\ga|t_\ga-s_\ga|+\ga|x_\ga-y_\ga|\leq C. 
\]   
Sending $\ga\to\infty$ in (\ref{eq:c6}) yields a contradiction. 
\end{proof}

\section{Skorokhod problem}

In this section we are concerned with the Skorokhod problem. 
We recall that $\R_+=(0,\,\infty)$ and hence $\ol\R_+=[0,\,\infty)$. 
We denote by $L^1_{\rm loc}(\ol\R_+,\,\R^k)$ (resp., 
$\AC_{\rm loc}(\ol\R_+,\,\R^k)$) the space of functions $v:\ol\R_+\to\R^k$ 
which are integrable (resp., absolutely continuous) 
on any bounded interval $I\subset\ol\R_+$.  

Given $x\in\ol\gO$ and 
$v\in L^1_{\rm loc}(\ol\R_+,\R^n)$, the Skorokhod problem is to seek for a pair of 
functions, $(\eta,\,l)\in\AC_{\rm loc}(\ol\R_+,\R^n)\tim L^1_{\rm loc}(\ol\R_+,\,\R)$, such that \smallskip
\begin{equation}\label{eq:s1} 
\left\{
\begin{aligned}
&\eta(0)=x, \ \ \ \ \ \eta(t)\in\ol\gO \ \ \hb{ for }t\in \ol\R_+, \\
&\dot\eta(t)+l(t)\gg(\eta(t))=v(t) \ \ \hb{ for a.e. }t\in \ol\R_+,  \\  
&l(t)\geq 0, \ \ \ l(t)=0 \ \hb{ if } \eta(t)\in\gO \ \ \hb{ for a.e. }t\in \ol\R_+. 
\end{aligned}
\right.
\end{equation}

Regarding the solvability of the Skorokhod problem, our main result is the following.

\advance\tolerance by 100

\begin{thm}
\label{thm:s0} Let $v\in L^1_{\rm loc}(\ol\R_+,\,\R^n)$ and $x\in\ol\gO$.   
Then there exits a pair $(\eta,\,l)\in \AC_{\rm loc}(\ol\R_+,\,\R^n)\tim 
L^1_{\rm loc}(\ol\R_+,\,\R)$ 
such that {\em(\ref{eq:s1})} holds. 
\end{thm} 

\advance\tolerance by -100

We are interested in ``regular'' solutions in the above theorem. 
See \cite{LionsSznitman_84} and references therein for more general viewpoints 
on the Skorokhod problem. The advantage of the above result is in that it applies to 
domains with $C^1$ boundary.

A natural question is the uniqueness of the solution $(\eta,\,l)$ in the above theorem. 
But we do now know if the uniqueness holds or not.

We first establish the following result. 

\advance\tolerance by 100

\begin{thm}\label{thm:s1}Let $v\in L^\infty(\ol\R_+,\,\R^n)$ and $x\in\ol\gO$.   
Then there exits a pair $(\eta,\,l)\in\Lip(\ol\R_+,\,\R^n)\tim L^\infty(\ol\R_+,\,\R)$ 
such that {\em(\ref{eq:s1})} holds. 
\end{thm}

\advance\tolerance by -100

We borrow some ideas from \cite{LionsSznitman_84} in the following proof.  

\begin{proof}  
We may assume that $\gg$ is defined on $\R^n$.  
Let $\psi\in C^1(\R^n)$ be such that 
$\psi(x)<0$ in $\gO$, $|D\psi(x)|>1$ for $x\in\gG$, $\psi(x)>0$ 
for $x\in\R^n\setminus \ol\gO$ and $\liminf_{|x|\to\infty}\psi(x)>0$. 
We can select a constant $\gd>0$ so that for $x\in\R^n$, 
\[
\gg(x)\cdot D\psi(x)\geq \gd|D\psi(x)| \ \ \hb{ if } 0\leq \psi(x)\leq \gd. 
\]
We set $q(x)=(\psi(x)\vee 0)\wedge\gd$ for $x\in\R^n$. 
Note that $q(x)=0$ for $x\in\ol\gO$, $q(x)>0$ for $x\in\R^n\setminus\ol\gO$, and 
$\gg(x)\cdot Dq(x)\geq \gd |Dq(x)|$ for a.e. $x\in\R^n$.   

Fix $\gep>0$ and $x\in\ol\gO$. We consider the initial value problem for the ODE
\begin{equation}\label{eq:s2}
\dot\xi(t)+\fr{1}{\gep}q(\xi(t))\gg(\xi(t))=v(t) \ \hb{ for a.e. }t\in \ol\R_+, \quad \xi(0)=x. 
\end{equation}
Here $\xi$ represents the unknown function. 
By the standard ODE theory, there is a unique solution $\xi\in 
C^1(\ol\R_+)$ of (\ref{eq:s2}).

Let $m\geq 2$. We multiply the ODE of (\ref{eq:s2}) by $mq(\xi(t))^{m-1}Dq(\xi(t))$, to get 
\[
\fr{\d }{\d t}q(\xi(t))^m+\fr m\gep q(\xi(t))^mDq(\xi(t))\cdot\gg(\xi(t))
=mq(\xi(t))^{m-1}Dq(\xi(t))\cdot v(t) \ \ \hb{a.e.}
\]
Fix any $T\in\R_+$.  Integrating over $[0,\,T]$, we get  
\begin{align*}
&q(\xi(T))^m-q(\xi(0))^m+\fr m\gep \int_0^Tq(\xi(s))^m\gg(\xi(s))\cdot Dq(\xi(s))\d s \\
&=m\int_0^T q(\xi(s))^{m-1}Dq(\xi(s))\cdot v(s)\d s.
\end{align*}
Here we have 
\[
\int_0^T q(\xi(s))^m\gg(\xi(s))\cdot Dq(\xi(s))\d s
\geq  \gd \int_0^Tq(\xi(s))^m|Dq(\xi(s))|\d s,
\]
and   
\begin{align*}
&\int_0^T q(\xi(s))^{m-1}Dq(\xi(s))\cdot v(s)\d s\\ 
&\leq \left(\int_0^Tq(\xi(s))^m|Dq(\xi(s)|\d s\right)^{1-\fr 1m}
\left(\int_0^T |v(s)|^m|Dq(\xi(s))|\d s\right)^{\fr 1m}. 
\end{align*}
Combining these, we get 
\begin{align}
&q(\xi(T))^m +\fr{m \gd}\gep \int_0^Tq(\xi(s))^m|Dq(\xi(s))|\d s\label{eq:s4}\\
&\leq m \left(\int_0^Tq(\xi(s))^m|Dq(\xi(s))|\d s\right)^{1-\fr 1m}
\left(\int_0^T |v(s)|^m|Dq(\xi(s))|\d s\right)^{\fr 1m}.\nonumber
\end{align}
From this we obtain 
\begin{equation}
\fr{ \gd}\gep \left(\int_0^Tq(\xi(s))^m|Dq(\xi(s))|\d s\right)^{\fr 1m}  
\leq 
\left(\int_0^T |v(s)|^m|Dq(\xi(s))|\d s\right)^{\fr 1m}\label{eq:s5}
\end{equation}
and
\[
q(\xi(T))^m\leq \left(\fr{\gep}{ \gd}\right)^{m-1}m\int_0^T |v(s)|^m|Dq(\xi(s))|\d s. 
\]
Hence, setting $C_0=\|Dq\|_{L^\infty(\R^n)}$, we deduce that 
\begin{equation}
q(\xi(t))^m\leq \left(\fr{\gep}{ a}\right)^{1-\fr 1m}mC_0T\|v\|_{L^\infty(0,T)}^m
 \ \ \hb{ for }t\in[0,\, T].\label{eq:s6}
\end{equation}
Henceforth we write $\xi_\gep$ for $\xi$, to indicate the dependence on $\gep$ of $\xi$.  
We see from (\ref{eq:s6}) 
that for any $T>0$, 
\begin{equation}\label{eq:s+1}
\lim_{\gep\to 0+}\max_{t\in[0,\,T]}\dist(\xi_\gep(t),\,\gO)=0. 
\end{equation}
Also, (\ref{eq:s6}) ensures that for each $T>0$ there is an $\gep_T>0$ such that 
$q(\xi_\gep(t))<\gd$ for $t\in[0,\,T]$. 

Now let $T>0$ and $0<\gep<\gep_T$. we have $q(\xi_\gep(s))
=\psi(\xi_\gep(s))\vee 0$ for all $t\in[0,\,T]$ and hence 
$q(\xi_\gep(t))^m|Dq(\xi_\gep(t))|=q(\xi_\gep(t))^m$ for a.e. $t\in(0,\,T)$. 
Accordingly, (\ref{eq:s5}) yields 
\[
\fr{ \gd}\gep \left(\int_0^Tq(\xi_\gep(s))^m\d s\right)^{\fr 1m}\leq C_0T^{\fr 1m}\|v\|_{L^\infty(0,T)}.  
\] 
Sending $m\to\infty$, we find that 
$
(\gd/\gep)\|q\circ\xi_\gep\|_{L^\infty(0,\,T)}\leq C_0\|v\|_{L^\infty(0,\,T)}$, 
and moreover 
\begin{equation}
\fr{ \gd}{\gep}\|q\circ\xi_\gep\|_{L^\infty(\R_+)}\leq C_0\|v\|_{L^\infty(\R_+)}. 
\label{eq:s7}
\end{equation}
We set $l_\gep=(1/\gep)q\circ \xi_\gep$. Due to (\ref{eq:s7}),  
we may choose a sequence $\gep_j\to 0+$ so that 
$l_{\gep_j} \to l$ weakly-star in $L^\infty(\R_+)$  
as $j\to\infty$ for a function 
$l\in L^\infty(\R_+)$. It is clear that $l(s)\geq 0$ for a.e. $s\in \R_+$.   

ODE (\ref{eq:s2}) together with (\ref{eq:s7}) guarantees 
that $\{\dot\xi_\gep\}_{\gep>0}$ is bounded in $L^\infty(\ol\R_+)$. Hence, 
we may assume as well that $\xi_{\gep_j}$ converges locally 
uniformly on $\ol\R_+$ to 
a function $\eta\in \Lip(\ol\R_+)$ as $j\to\infty$. 
It is then obvious that $\eta(0)=x$ and the pair $(\eta,\,l)$  
satisfies 
\[
\eta(t)+\int_0^t \big(l(s)\gg(\xi(s))-v(s)\big)\d s=0 \ \ \hb{ for }t>0,
\]
from which we get
\[
\dot\eta(t)+l(t)\gg(\eta(t))=v(t) \ \ \hb{ for a.e. }t\in \ol\R_+. 
\]
It follows from (\ref{eq:s+1}) that $\eta(t)\in\ol\gO$ for $t\geq 0$.

In order to show that the pair $(\eta,\,l)$ is a solution of (\ref{eq:s1}), 
we need only to prove that for a.e. $t\in \ol\R_+$, $l(t)=0$ if $\eta(t)\in\gO$.   
Set $A=\{t\geq 0\mid \eta(t)\in\gO\}$. It is clear that $A$ is an open subset of 
$[0,\,\infty)$.  We can choose a sequence $\{I_k\}_{k\in\N}$ of closed intervals of $A$ 
such that $A=\bigcup_{k\in\N}I_k$. Note that for each $k\in\N$,  the set $\eta(I_k)$ 
is a compact 
subset of $\gO$ and the convergence of $\{\xi_{\gep_j}\}$ to  
$\eta$ is uniform on $I_k$. Hence, 
for any fixed $k\in\N$, we may choose 
$J\in\N$ so that $\xi_{\gep_j}(t)\in\gO$ for all $t\in I_k$ and $j\geq J$. 
From this, we have 
$q(\xi_{\gep_j}(t))=0$ for $t\in I_k$ and $j\geq J$. Moreover, 
in view of the weak-star convergence of $\{l_{\gep_j}\}$, we find that for any 
$k\in\N$,  
\[
\int_{I_k}l(t)\d t=\lim_{j\to\infty}
\int_{I_k}\fr{1}{\gep_j}q(\xi_j(t))^m\d t=0,
\]  
which yields $l(t)=0$ for a.e. $t\in I_k$. Since $A=\bigcup_{k\in\N} I_k$, we see that 
$l(t)=0$ a.e. in $A$. 
The proof is complete.
\end{proof}

For $x\in\ol\gO$, let $\SP(x)$ denote 
the set of all triples 
\[(\eta,v,l)\in \AC_{\rm loc}(\ol\R_+,\R^n)\tim L^1_{\rm loc}(\ol\R_+,\R^n)
\tim L^1_{\rm loc}(\ol\R_+)\] which satisfies (\ref{eq:s1}). 
We set $\SP=\bigcup_{x\in\ol\gO}\SP(x)$.  

We remark that for any $x,y\in\ol\gO$ and $0<T<\infty$, there exists 
a triple $(\eta,v,l)\in\SP(x)$ such that $\eta(T)=y$. 
Indeed, given $x,y\in\ol\gO$ and $0<T<\infty$, we choose a 
curve $\eta\in\Lip([0,\,T],\ol\gO)$ so that $\eta(0)=x$ and $\eta(T)=y$. 
The existence of such a curve is guaranteed since $\gO$ is a domain  
and has the $C^1$ regularity.    
We extend the domain of definition of $\eta$ to 
$\ol\R_+$ by setting $\eta(t)=y$ for $t>T$. Now, if we set $v(t)=\dot\eta(t)$ 
and $l(t)=0$ for $t\geq 0$, we have $(\eta,v,l)\in\SP(x)$, which has the property, 
$\eta(T)=y$. 
Here and henceforth, for interval $I$, we denote by $\Lip(I,\ol\gO)$ the set of 
those $\eta\in \Lip(I,\R^n)$ such that $\eta(t)\in\ol\gO$ for $t\in I$. 
We use such notation for other spaces of functions having values in $\ol\gO\subset \R^n$ as well. 

We note also that problem (\ref{eq:s1}) has the following {\em semi-group} property: 
for any $(x,t)\in\ol\gO\tim\R_+$ and $(\eta_1,\,v_1,\,l_1),\, (\eta_2,\,v_2,\,l_2)\in\SP$, 
if $\eta_1(0)=x$ and $\eta_2(0)=\eta_1(t)$ hold and if $(\eta,\,v,\,l)$ is defined on $\ol\R_+$ by
\[
(\eta(s),\,v(s),\,l(s))=\left\{
\begin{aligned}
&(\eta_1(s),\,v_1(s),\,l_1(s)) \quad&&\hb{ for }s\in[0,\,t),\\
&(\eta_2(s-t),\,v_2(s-t),\,l_2(s-t)) \quad&& \hb{ for }s\in[t,\,\infty),
\end{aligned}
\right.
\] 
then $(\eta,\,v,\,l)\in\SP(x)$.

\begin{prop}\label{thm:s2}
There is a constant $C>0$, depending only on $\gO$ and $\gg$, such that for $(\eta,\,v,\,l)\in\SP$, 
\[ 
|\dot\eta(s)|\vee l(s)\leq C|v(s)| \ \ \hb{ for a.e. }s\geq 0. 
\]
\end{prop}

An immediate consequence of the above proposition is that for $(\eta,\,v,\,l)\in\SP$, 
if $v\in L^p(\R_+,\,\R^n)$ (resp., $v\in L^p_{\rm loc}(\R_+,\,\R^n)$), 
with $1\leq p\leq \infty$, then  $(\eta,\,l)\in L^p(\R_+,\,R^{n+1})$ 
(resp., $(\eta,\,l)\in L^p_{\rm loc}(\R_+,\,\R^{n+1})$).

\begin{proof}
Thanks to hypothesis (A3), 
there is a constant $\gd_0>0$ such that $\nu(x)\cdot\gg(x)\geq \gd_0$ 
for $x\in\gG$. 
Let $(\eta,v,l)\in\SP$. According to the $C^1$ regularity of $\gO$, there 
is a function $\psi\in C^1(\R^n)$ such that  
\[
\gO=\{x\in\R^n\mid \psi(x)<0\} \ \ \ \hb{ and } \ \ \ D\psi(x)\not=0 \ \ \hb{ for }x\in\gG. 
\]
Noting that $\psi(\eta(s))\leq 0$ for all $s\geq 0$, 
we find that for any $s>0$, if $\eta(s)\in\gG$ and $\eta$ is differentiable at $s$, then 
\[
0=\fr{\d}{\d s}\psi(\eta(s))=D\psi(\eta(s))\cdot\dot\eta(s). 
\]
Hence, noting that $D\psi(\eta(s))$ is parallel to $\nu(\eta(s))$, we see that  
$\nu(\eta(s))\cdot \dot\eta(s)=0$. 

Let $s>0$ be such that $\eta(s)\in\gG$, $\dot\eta(s)$ exists,    
$\dot\eta(s)+l(s)\gg(\eta(s))=v(s)$, $l(s)\geq 0$ 
and $\nu(\eta(s))\cdot\dot\eta(s)=0$. We see immediately that
$l(s)\gg(\eta(s))\cdot\nu(\eta(s))=v(s)\cdot\nu(\eta(s))$. Hence, we get
\[
\gd_0\, l(s)\leq v(s)\cdot \nu(\eta(s))\leq |v(s)|
\]
and $l(s)\leq \gd_0^{-1}|v(s)|$  for a.e. $s\geq 0$. We also have
\[
\left|\dot\eta(s)\right|\leq |v(s)|+\|\gg\|_{\infty}|l(s)|\leq \Big(1+\fr{\|\gg\|_{\infty}}{\gd_0}\Big)
|v(s)| \ \ \hb{ for a.e. }s\geq 0. \qedhere
\]
\end{proof}

Let $\cF$ be a subset of $L^1(I,\R^m)$, where $I\subset\R$ is an interval. 
We recall that $\cF$ is said to be {\em uniformly integrable}   
if for any $\gep>0$ there is a $\gd>0$ such that for any $f\in\cF$, 
\[
\Big|\int_B f(s)\d s\Big|<\gep
 \ \  
\hb{ whenever } \ B\subset I\hb{ is measurable}  \ \hb{ and }  \  |B|<\gd. 
\] 
Here $|B|$ denotes the Lebesgue measure of $B\subset\R$. 

\begin{prop}\label{thm:s3} Let $\{(\eta_k,\,v_k,\,l_k)\}_{k\in\N}\subset\SP$. 
Assume that $\{|v_k|\}$ is uniformly integrable on every intervals  
$[0,\,T]$, with $0<T<\infty$. Then there exist a subsequence $\{\eta_{k_j},\,v_{k_j},\,l_{k_j}\}_{j\in\N}$ of 
$\{\eta_k,\,v_k,\,l_k\}$ and a $(\eta,\,v,\,l)\in\SP$ such that 
\begin{align*}
& \eta_{k_j}(t) \to \eta(t) \ \ \hb{ uniformly on }[0,\,T], \\
& \dot\eta_{k_j}\d t \to \dot\eta\d t \ \ \hb{ weakly-star in } C([0,\,T],\,\R^n)^*,\\
& v_{k_j}\d t \to v\d t \ \ \hb{ weakly-star in } C([0,\,T],\,\R^n)^*,\\
& l_{k_j}\d t \to l\d t \ \ \hb{ weakly-star in } C([0,\,T])^*
\end{align*} 
for every $T>0$.  
\end{prop}

In the above proposition, we denote by $X^*$ the dual space of the Banach space $X$. 
Regarding notation in the above proposition, we remark 
that the weak-star convergence in $C([0,\,T])^*$ or $C([0,\,T],\,\R^n)^*$ 
is usually stated as the weak convergence of measures.  

\begin{proof} By Proposition \ref{thm:s2}, there is a constant $C_0>0$ such that 
for $k\in\N$, 
\begin{equation}\label{eq:s8}
\left|\dot\eta_k(s)\right|\vee l_k(s)\leq C_0|v_k(s)| \ \ \hb{ for a.e. }s\geq 0. 
\end{equation}
It follows from this that the sequences $\{\left|\dot\eta_k\right|\}$ 
and $\{l_k\}$ are uniformly integrable on the intervals $[0,\,T]$, $0<T<\infty$.  
If we set 
\[
V_k(t)=\int_0^t v_k(s)\d s  \ \ \hb{ and } \ \ L_k(t)=\int_0^t l_k(s)\d s \ \ \hb{ for }t\geq 0, 
\]
then the sequences $\{\eta_k\}$, $\{V_k\}$ and $\{L_k\}$ are equi-continuous and uniformly bounded 
on the intervals $[0,\,T]$, $0<T<\infty$. We may therefore choose 
an increasing sequence $\{k_j\}\subset\N$ so that the sequences 
$\{\eta_{k_j}\}$, $\{V_{k_j}\}$ and $\{L_{k_j}\}$ converge, as $j\to\infty$, uniformly 
on every finite interval $[0,\,T]$,  $0<T<\infty$, to some functions 
$\eta\in C(\ol\R_+,\,\ol\gO)$, $V\in C(\ol\R_+,\,\R^n)$ and $L\in 
C(\ol\R_+)$. The uniform integrability of the sequences 
$\{|\dot\eta_k|\}$, $\{|v_k|\}$ and $\{l_k\}$ implies that 
the functions $\eta$, $V$ and $L$ are absolutely continuous on every finite 
interval $[0,\,T]$, $0<T<\infty$.  

Fix any $0<T<\infty$. 
The uniform integrability of the sequences $\{\left|\dot\eta_{k}\right|\}$, 
$\{\left|v_{k}\right|\}$ and $\{l_{k}\}$ guarantees that the sequences 
$\{\dot\eta_k \d s\}$, $\{v_k\d s\}$ and $\{l_k\d s\}$ 
of measures on $[0,\,T]$ are bounded. 
That is, we have 
\[
\sup_{k\in\N}\int_0^T(|\dot\eta_k(s)|+|v_k(s)|+l_k(s))\d s<\infty.
\]
Hence we may assume without loss of generality that as $j\to\infty$, 
\begin{align*}
& \dot\eta_{k_j}\d s \to  \mu_1 \ \ \hb{ weakly-star in } C([0,\,T],\,\R^n)^*,\\
& v_{k_j}\d t \to \mu_2 \ \ \hb{ weakly-star in } C([0,\,T],\,\R^n)^*,\\
& l_{k_j}\d t \to \mu_3 \ \ \hb{ weakly-star in } C([0,\,T])^*
\end{align*} 
for some regular Borel measures $\mu_1$, $\mu_2$ and $\mu_3$ of bounded variations
on $[0,\,T]$. Then, for any $\phi\in C^1([0,\,T],\,\R^n)$, using integration by parts twice, 
we get
\begin{align*}
\int_0^T \phi(s)\mu_1(\d s) 
&=\lim_{j\to\infty}\int_0^T\phi(s)\dot\eta_{k_j}(s)\d s\\
&=\lim_{j\to\infty}\Big(\Big[\phi\eta_{k_j}\Big]_0^T-\int_0^T\phi'(s)\eta_{k_j}(s)\d s\Big)\\
&=\Big[\phi\eta\Big]_0^T-\int_0^T\phi'(s)\eta(s)\d s
=\int_0^T\phi(s)\dot\eta(s)\d s.  
\end{align*}  
By the density of $C^1([0,\,T],\,\R^n)$ in $C([0,\,T],\,\R^n)$, we find that 
\[
\int_0^T \phi(s)\mu_1(\d s)=\int_0^T\phi(s)\dot\eta(s)\d s,
\]
which shows that $\mu_1=\dot\eta\d s$ on $[0,\,T]$. Similarly we see that 
$\mu_2=\dot V\d s$ and $\mu_2=\dot L\d s$. Thus, setting $v=\dot V$ and $l=\dot L$, 
we have as $j\to\infty$    
\begin{align*}
& \dot\eta_{k_j}\d s \to  \dot\eta \d s \ \ \hb{ weakly-star in } C([0,\,T],\,\R^n)^*,\\
& v_{k_j}\d t \to v\d s \ \ \hb{ weakly-star in } C([0,\,T],\,\R^n)^*,\\
& l_{k_j}\d t \to l\d s \ \ \hb{ weakly-star in } C([0,\,T])^*. 
\end{align*} 
Note here that the above weak-star convergence is valid for every $0<T<\infty$. 

Since 
\[
\dot\eta_k(s)+l_k(s)\gg(\eta_k(s))=v_k(s) \ \ \hb{ for a.e. }s\geq 0,
\]  
integrating this over $[0,\,t]$, $0<t<\infty$ and sending $k=k_j$ as $j\to\infty$, we get
\[
\eta(t)-\eta(0)+\int_0^t l(s)\gg(\eta(s))\d s=\int_0^t v(s)\d s \ \ \hb{ for }t>0,
\]
which ensures that $\dot\eta(s)+l(s)\gg(\eta(s))=v(s)$ for a.e. $s\geq 0$. 
It is obvious that $\eta(s)\in\ol\gO$ for $s\geq 0$. Finally, we argue as in the last part of the proof of Theorem \ref{thm:s1}, to find that for a.e. $s\in\R_+$, 
$l(s)=0$ if $\eta(s)\in\gO$.   
 The proof is complete.  
\end{proof}

\begin{proof}[Proof of Theorem {\em\ref{thm:s0}}] Fix any $x\in\ol\gO$ and  
$v\in L^1_{\rm loc}(\ol\R_+,\,\R^n)$. 
In view of the semi-group property 
of problem (\ref{eq:s1}), we may assume that $v(s)=0$ for $s\geq 1$, so that 
$v\in L^1(\ol\R_+,\,\R^n)$. We define the sequence $\{v_k\}_{k\in\N}\subset 
L^\infty(\ol\R_+,\,\R^n)$ by 
\[
v_k(s)
=
\left\{
\begin{aligned}
&v(s)\quad&&\hb{ if }|v(s)|\leq k,\\
&0&&\hb{ otherwise}. 
\end{aligned}
\right.
\] 
Since $|v_k(s)|\leq |v(s)|$ for $s\geq 0$, we see that the sequence 
$\{|v_k|\}$ is uniformly integrable on $\ol\R_+$. 
According to Theorem \ref{thm:s1}, there is a sequence 
$\{(\eta_k,\,l_k)\}\subset \Lip(\ol\R_+,\,\R^n)\tim L^\infty(\ol\R_+,\,\ol\R_+)$ 
such that $(\eta_k,\,v_k,\,l_k)\in\SP(x)$ for all $k\in\N$.  
Then applying Proposition \ref{thm:s3}, we deduce that there is a 
$(\eta,\,l)\in \AC_{\rm loc}(\ol\R_+,\,\R^n)\tim 
L^1_{\rm loc}(\ol\R_+,\,\ol\R_+)$ such that 
$(\eta,\,v,\,l)\in\SP(x)$. 
\end{proof}

\section{Cauchy problem with the Neumann type boundary condition} 

In this section we introduce the value function of an optimal control problem 
associated with the initial-boundary value problem (\ref{eq:i3})--(\ref{eq:i5}), and show that it is a (unique) solution of problem (\ref{eq:i3})--(\ref{eq:i5}).  
     
We define the function $L\in\LSC(\ol\gO\tim\R^n,\,\R\cup\{\infty\})$, called the 
Lagrangian of $H$, by 
\[
L(x,\xi)=\sup_{p\in\R^n}\big(\xi\cdot p-H(x,p)\big). 
\]      
The value function $w$ of the optimal control with the dynamics given by (\ref{eq:s1}), 
the running cost $(L,\,g)$ and the pay-off $u_0$ is given by 
\begin{equation}\label{eq:cn1}
\begin{split}
w(x,t)=\inf\Big\{&\int_0^t \big(L(\eta(s),-v(s))+g(\eta(s))l(s)\big)\d s \\ 
&+u_0(\eta(t)) 
\mid 
(\eta,v,l)\in\SP(x)\Big\} 
\quad \hb{ for }(x,t)\in\ol\gO\tim\R_+. 
\end{split}
\end{equation}
Under our hypotheses, the Lagrangian $L$ may take the value $\infty$ and, on the other hand, 
there is 
a constant $C_0>0$ such that $L(x,\xi)\geq -C_0$ for $(x,\,\xi)\in\ol\gO\tim\R^n$. 
Thus, it is reasonable to interpret 
\[
\int_0^t L(\eta(s),-v(s))\d s=\infty 
\]
if the function: $s\mapsto L(\eta(s),-v(s))$
is not integrable, 
which we adopt here.  

It is well-known that (and also easily seen) the value function $w$ satisfies 
the dynamic programming principle
\begin{align*}
w(x,s+t)=&\,\inf\Big\{\int_0^t \big(L(\eta(\tau),-v(\tau))+g(\eta(\tau))l(\tau)\big)\d \tau+w(\eta(t),\,s) \mid\\
&\, \qquad\qquad (\eta,v,l)\in\SP(\eta(s))\Big\}  
\ \ \hb{ for }x\in\ol\gO \ \hb{ and } \ t, s\in\R_+. 
\end{align*}

\begin{thm}\label{thm:cn1} The value function $w$ is continuous on $\ol\gO\tim\R_+$ 
and it is a solution of {\em(\ref{eq:i3})--(\ref{eq:i4})}, with $a:=0$. Moreover, $w$ 
satisfies {\em(\ref{eq:i5})} in the sense that 
\[
\lim_{t\to 0+}w(x,t)=u_0(x) \ \ \hb{ uniformly for }x\in\ol\gO. 
\]
\end{thm}

The above theorem clearly ensures the existence of a solution of (\ref{eq:i3})--(\ref{eq:i5}), 
with $a:=0$. 
This together with Theorem \ref{thm:c4}, with $U:=\R^n$, establishes 
the unique existence of a solution of (\ref{eq:i3})--(\ref{eq:i5}), with $a:=0$.  
For the solvability of stationary and evolution problem for 
HJ-Jacobi equations, we refer to \cite{Lions-Neumann_85, LionsTrudinger_86,  BarlesLions_91, DupuisIshii_90, Barles_93, UG}. 

Another aspect of the theorem above is that it gives a variational formula for  
the unique solution of    
(\ref{eq:i3})--(\ref{eq:i5}), with $a:=0$. 
This is a classical observation on the value functions in optimal control, and, in this regard,  
we refer for instance to \cite{Lions-Neumann_85, LionsTrudinger_86}.  

The variational formula (\ref{eq:cn1}) is sometimes 
called the Lax-Oleinik formula. The formula (\ref{eq:cn1}) still  
valid for the solution of (\ref{eq:i3})--(\ref{eq:i5}) with general $a\in\R$ if 
one replaces the Lagrangian $L(x,\xi)$ by $L(x,\xi)+a$.

For the proof of Theorem \ref{thm:cn1}, we need the following three lemmas. 
In what follows we {\em always} assume that $a=0$ in (\ref{eq:i3}). 
We set $Q=\ol\gO\tim\R_+$.

\begin{lem}\label{thm:cn2}
Let $\psi\in C^1(\ol Q)$ be a classical subsolution of {\em 
(\ref{eq:i3})--(\ref{eq:i4})}.   
Assume that $\psi(x,0)\leq u_0(x)$ for $x\in\ol\gO$.   
Then 
$w\geq \psi$ on $\ol Q$. 
\end{lem}

\begin{proof}
Let $(x,t)\in Q$ and $(\eta,\,v,\,l)\in\SP(x)$. 
We have 
\begin{align*}
&\psi(\eta(t),0)-\psi(\eta(0),t)=
\int_0^t\fr{\d}{\d s}\psi(\eta(s),t-s)\d s \\
&=\int_0^t \big(D\psi(\eta(s),t-s)\cdot \dot\eta(s)-\psi_t(\eta(s),t-s)\big)\d s\\
&=\int_0^t \big(D\psi(\eta(s),t-s)\cdot (v(s)-l(s)\gg(\eta(s)))
-\psi_t(\eta(s),t-s)\big)\d s.
\end{align*}
Now, using the subsolution property of $\psi$ 
and the inequality $\psi(\cdot,0)\leq u_0$, we get 
\begin{align*}
&\psi(x,t)-u_0(\eta(t))\\
&\leq\int_0^t\big(-D\psi(\eta(s),t-s)\cdot v(s)
+l(s)D\psi(\eta(s))\cdot\gg(\eta(s)) 
+\psi_t(\eta(s),t-s)\big)\d s\\
&\leq\int_0^t\big(H(\eta(s),D\psi(\eta(s),t-s))+L(\eta(s),-v(s))
+l(s)D\psi(\eta(s))\cdot\gg(\eta(s))\cr
&\quad +\psi_t(\eta(s),t-s)\big)\d s\\
&\leq\int_0^t\big(
L(\eta(s),-v(s))+l(s)g(\eta(s))
\big)\d s. 
\end{align*}
Thus we conclude that $\psi(x,t)\leq w(x,t)$. 
\end{proof}

\begin{lem} \label{thm:cn2+1}
For any $\gep>0$ there is a constant $C_\gep>0$ such that 
$w(x,t)\geq u_0(x)-\gep-C_\gep t$ for $(x,t)\in Q$. 
\end{lem}

\begin{proof}We fix any $\gep>0$ and choose a function $u_0^\gep\in C^1(\ol\gO)$ so that 
$|u_0(x)-u_0^\gep(x)|\leq\gep$ for $x\in\ol\gO$.   
We choose a function $\psi_0\in C^1(\R^n)$ so that 
$\gO=\{x\in\R^n\mid \psi_0(x)<0\}$ and $D\psi_0(x)\not=0$ for $x\in\gG$.  
By multiplying $\psi_0$ by a positive constant, we may find a function 
$\psi^\gep\in C^1(\ol\gO)$ so that 
\[
\gg(x)\cdot D(u_0^\gep+\psi^\gep)(x)\geq g(x) \ \ \hb{ for }x\in\gG.\]
Next, approximating the function: $r\mapsto 
(-\gep)\vee(\gep\wedge r)$ on $\R$ by a smooth function, 
we build a function $\gz_\gep\in C^1(\R)$ so that $|\gz_\gep(r)|\leq\gep$ 
for $r\in\R$ and $\gz_\gep'(0)=1$. Note that 
$D(\gz_\gep\circ \psi^\gep)(x)=D\psi^\gep(x)$ for $x\in\gG$ 
and $|u_0(x)-u_0^\gep(x)-\gz_\gep\circ \psi^\gep(x)|\leq 2\gep$ for $x\in\ol\gO$. 
We choose a constant $C_\gep>0$ so that 
\[
H(x,D(u_0^\gep+\gz_\gep\circ \psi^\gep)(x))\leq C_\gep \ \ \hb{ for }x\in\ol\gO.  
\] 
Finally we define the function $\phi^\gep\in C^1(\ol Q)$ by 
\[
\phi^\gep(x,t)=-2\gep+u_0^\gep(x)+\gz_\gep\circ \psi^\gep(x)-C_\gep t, 
\]  
and observe that $\phi^\gep$ is a classical subsolution of (\ref{eq:i3}), (\ref{eq:i4}) 
and that $\phi^\gep(x,0)\leq u_0(x)$ for $x\in\ol\gO$. By Lemma \ref{thm:cn2}, we get 
$\phi^\gep(x,t)\leq w(x,t)$ for $(x,t)\in Q$. Hence, 
we obtain $w(x,t)\geq u_0(x)-4\gep-C_\gep t$ for all $(x,t)\in Q$.  
\end{proof}

\advance \tolerance by 100

\begin{lem}\label{thm:cn3}
There is a constant $C>0$ such that 
$w(x,t)\leq u_0(x)+Ct$ for $(x,t)\in Q$.
\end{lem}

\advance \tolerance by -100

\begin{proof} Let $(x,t)\in Q$. Set $\eta(s)=x$, $v(s)=0$ and $l(s)=0$ for 
$s\geq 0$. Then $(\eta,v,l)\in\SP(x)$. Hence, we have 
\[
w(x,t)\leq u_0(x)+\int_0^tL(x,0)\d s=u_0(x)+tL(x,0)\leq u_0(x)-t\min_{p\in\R^n}H(x,p).
\]
Setting $C=-\min_{\ol\gO\tim\R^n}H$, we get $w(x,t)\leq u_0(x)+Ct$. 
\end{proof}

\begin{lem}\label{thm:cn4}
Let $t>0$, $x\in\ol\gO$, $\phi\in C^1(\ol\gO\tim[0,\,t])$ and $\gep>0$. 
Then there is a triple $(\eta,v,l)\in\SP(x)$ such that for a.e. $s\in(0,\,t)$, 
\[
H(\eta(s),D\phi(\eta(s),t-s))+L(\eta(s),-v(s))\leq\gep-v(s)\cdot D\phi(\eta(s),\,t-s). 
\]
\end{lem}
 
We postpone the proof of the above lemma and give now 
the proof of Theorem \ref{thm:cn1}. 

\def\beeq{\begin{equation}}
\def\eneq{\end{equation}}
\def\x{\hat x}
\def\t{\hat t}

\begin{proof}[Proof of Theorem {\em\ref{thm:cn1}}] By Lemmas \ref{thm:cn2+1} and \ref{thm:cn3}, 
there is a constant $C>0$ and for each $\gep>0$ a constant $C_\gep>0$ such that 
\[
-\gep-C_\gep t\leq w(x,t)-u_0(x)\leq Ct \ \ \hb{ for all }(x,t)\in Q.
\]
This shows that $w$ is a real-valued function on $Q$ and that 
\begin{equation} \label{eq:cn2}
\lim_{t\to 0+}w(x,t)=u_0(x) \ \ \hb{ uniformly for }x\in\ol\gO. 
\end{equation}

We next prove that 
$w$ is a subsolution of (\ref{eq:i3}), (\ref{eq:i4}).   
Let $(\x,\t)\in Q$ and $\phi\in C^1(\ol Q)$. Assume that 
$w^*-\phi$ attains a strict maximum at $(\x,\t)$. We need to show that if $\x\in\gO$, then  
\[
\phi_t(\x,\t)+H(\x,D\phi(\x,\t))\leq 0, 
\]
and if $\x\in\gG$, then either 
\begin{equation}
\phi_t(\x,\t)+H(\x,D\phi(\x,\t))\leq 0 \ \ \hb{ or } \ \ \gg(\x)\cdot D\phi(\x,\t)\leq g(\x). 
\label{eq:cn3}
\end{equation}

We are here concerned only with the case where $\x\in\gG$. The other case can be treated 
similarly. To prove (\ref{eq:cn3}), 
we argue by 
contradiction. Thus we suppose that (\ref{eq:cn3}) were false. 
We may choose an $\gep\in(0,\,1)$ so that 
$\t-2\gep>0$ and for $(x,t)\in \big(\ol\gO\cap B(\x,\,2\gep)\big)\tim[\t-2\gep,\,\t+2\gep]$, 
\beeq
\phi_t(x,t)+H(x,D\phi(x,t))\geq 2\gep \ \ \hb{ and  } \ \ 
\gg(x)\cdot D\phi(x,t)-g(x)\geq2\gep,
\label{eq:cn4}
\eneq  
where $\gg$ and $g$ are assumed to be defined and continuous on $\ol\gO$. 
We may assume that $(w^*-\phi)(\x,\t)=0$. 
Set
\[
B=\big(\pl B(\x,2\gep)\tim [\t-2\gep,\t+2\gep] \cup B(\x,2\gep)\tim\{\t-2\gep\}\big)\, \cap\,\ol Q,
\]
and 
$m=-\max_{B}(w^*-\phi)$.  Note that $m>0$ 
and
$w(x,t)\leq \phi(x,t)-m$ for $(x,t)\in B$. 
We choose a point $(\bar x,\bar t)\in \ol\gO\cap B(\x,\gep)
\tim[\t-\gep,\,\t+\gep]$ so that $(w-\phi)(\bar x,\bar t)>-\gep^2\wedge m$.  
We apply Lemma \ref{thm:cn4}, to find a triple $(\eta,v,l)\in\SP(\bar x)$ such that 
for a.e. $s\geq 0$, 
\beeq
H(\eta(s),D\phi(\eta(s),\bar t-s))+L(\eta(s),-v(s))\leq \gep-v(s)\cdot D\phi(\eta(s),\bar t-s)
\label{eq:cn5}
\eneq
Note that $\gs:=\bar t-(\t-2\gep)\geq \gep$ and $\dist(\bar x, \pl B(\x,2\gep))\geq \gep$. 
Set 
\[
S=\{s\in[0,\,\gs]\mid \eta(s)\in \pl B(\x,2\gep)\} \ \ 
\hb{ and } \ \  \tau=\inf S.
\]
We consider first the case where $\tau=\infty$, i.e., the case $S=\emptyset$. 
By the dynamic programming 
principle, we have
\begin{align*}
\phi(\bar x,\bar t)<&\, w(\bar x,\bar t)+\gep^2 \cr
\leq&\,  \int_0^\gs \big(L(\eta(s),-v(s))+g(\eta(s))l(s)\big)\d s+w(\eta(\gs),\bar t-\gs)
+\gep^2\cr 
\leq&\,  \int_0^\gs \big(L(\eta(s),-v(s))+g(\eta(s))l(s)+\gep\big)\d s+\phi(\eta(\gs),\bar t-\gs).
\end{align*}
Hence, we obtain 
\begin{align*}
0<&\, \int_0^\gs\big(L(\eta(s),-v(s))+g(\eta(s))l(s)+\gep
+\fr{\d}{\d s}\phi(\eta(s),\bar t-s)
\big)\d s\cr 
\leq &\, \int_0^\gs\big(L(\eta(s),-v(s))+g(\eta(s))l(s)+\gep\cr 
&\, \quad +D\phi(\eta(s),\bar t-s)\cdot\dot\eta(s)-\phi_t(\eta(s),\bar t-s)
\big)\d s\cr 
\leq &\, \int_0^\gs\big(L(\eta(s),-v(s))+g(\eta(s))l(s)+\gep\cr 
&\, \quad +D\phi(\eta(s),\bar t-s)\cdot(v(s)-l(s)\gg(\eta(s))-\phi_t(\eta(s),\bar t-s)
\big)\d s.\cr 
\end{align*}
Now, using (\ref{eq:cn5}) and (\ref{eq:cn4}), we get 
\begin{align*}
0<&\, \int_0^\gs\big(2\gep-H(\eta(s),D\phi(\eta(s),\bar t-s))+g(\eta(s))l(s)\cr 
&\, \quad -l(s)D\phi(\eta(s),\bar t-s)\cdot \gg(\eta(s))-\phi_t(\eta(s),\bar t-s)
\big)\d s\cr 
\leq&\, \int_0^\gs l(s)\big(g(\eta(s))-\gg(\eta(s))\cdot D\phi(\eta(s),\bar t-s)
\big)\d s\leq 0, 
\end{align*}
which is a contradiction. 

Next we consider the case where $\tau<\infty$. Observe that $\tau>0$ and  
\begin{align*}
\phi(\bar x,\bar t)<&\, w(\bar x,\bar t)+m \cr
\leq&\, \int_0^\tau 
\big(L(\eta(s),-v(s))+g(\eta(s))l(s)\big)\d s+w(\eta(\tau),\bar t-\tau)+m\cr
\leq&\,  \int_0^\tau 
\big(L(\eta(s),-v(s))+g(\eta(s))l(s)\big)\d s+\phi(\eta(\tau),\bar t-\tau).
\end{align*}
Using (\ref{eq:cn5}) and (\ref{eq:cn4}) as before, we compute that
\begin{align*}
0<&\, \int_0^\tau 
\big(L(\eta(s),-v(s))+g(\eta(s))l(s) -\phi_t(\eta(s),\bar t-s)\cr
&\,\quad
+D\phi(\eta(s),\bar t-s)\cdot v(s)-l(s)\gg(\eta(s))\cdot D\phi(\eta(s),\bar t-s)
\big)\d s \cr
\leq &\,  \int_0^\tau \big\{\gep-H(\eta(s),D\phi(\eta(s),\bar t-s))
-\phi_t(\eta(s),\bar t-s)\cr
&\,\quad 
+l(s)[g(\eta(s))-\gg(\eta(s))\cdot D\phi(\eta(s),\bar t-s)]\big\}\d s
<0, 
\end{align*}
which is again a contradiction. Thus, we conclude that $w$ 
is a subsolution of (\ref{eq:i3}), (\ref{eq:i4}). 

Now, we turn to the proof of the supersolution property of $w$. 
Let $\phi\in C^1(\ol Q)$ and $(\x,\t)\in\ol\gO\tim(0,\,\infty)$. 
Assume that $w_*-\phi$ attains a strict minimum at $(\x,\t)$. 
We show that if $\x\in\gO$, then 
\[
\phi_t(\x,\t)+H(\x,D\phi(\x,\t))\geq 0,  
\]
and if $\x\in\gG$, then 
\beeq
\phi_t(\x,\t)+H(\x,D\phi(\x,\t))\geq 0 \ \ \hb{ or } \ \ 
\gg(\x)\cdot D\phi(\x,\t)\geq g(\x).\label{eq:cn6}
\eneq
We only consider the case where $\x\in\gG$, and leave it to the reader to 
check the details in the other case. 
To show (\ref{eq:cn6}), we suppose by contradiction that (\ref{eq:cn6}) were 
false. That is, we have 
\[
\phi_t(\x,\t)+H(\x,D\phi(\x,\t))<0 \ \ \hb{ and } \ \ 
\gg(\x)\cdot D\phi(\x,\t)-g(\x)<0.
\]
There is an $\gep>0$ such that  
\[
\phi_t(x,t)+H(x,D\phi(x,t))<0 \ \ \hb{ and } \ \ 
\gg(x)\cdot D\phi(x,t)-g(x)<0 \ \ \hb{ for }(x,t)\in R\cap\ol Q,
\] 
where $R:=B(\x,2\gep)\tim[\t-2\gep,\t+2\gep]$.
Here we may assume that $\t-2\gep>0$ and $(u_*-\phi)(\x,\t)=0$. 
Set 
\[
m:=\min_{\ol Q\cap \pl R}(u_*-\phi)\ (>0).
\]
We may choose a point $(\bar x,\bar t)\in\ol Q$ so that 
$(u_*-\phi)(\bar x,\bar t)<m$, $|\bar x-\x|<\gep$ and $|\bar t-\t|<\gep$.  
We select a triple $(\eta,v,l)\in\SP(\bar x)$ so that 
\[
u(\bar x,\bar t)+m>\int_0^{\bar t}\big(L(\eta(s),-v(s))+g(\eta(s))l(s)\big)\d s
+u_0(\eta(\bar t)). 
\]
We set 
\[
\tau=\min\{s\geq 0\mid (\eta(s),\bar t-s)\in\pl R\}. 
\]
It is clear that $\tau>0$, $\eta(s)\in R\cap\ol Q$ for 
$s\in[0,\,\tau]$ and,  
if $|\eta(\tau)-\x|<2\gep$, then $\tau=\bar t-(\t-2\gep)>\gep$. 
Accordingly, we have  
\begin{align*}
\phi(\bar x,\bar t)+m>&\, \int_0^\tau \big(
L(\eta(s),-v(s))+g(\eta(s))l(s)
\big)\d s+u(\eta(\tau),\bar t-\tau)\cr
\geq&\, \int_0^\tau \big(
L(\eta(s),-v(s))+g(\eta(s))l(s) 
\big)\d s+\phi(\eta(\tau),\bar t-\tau)+m.
\end{align*}
Hence, we get
\begin{align*}
0>&\,\int_0^\tau\big(
L(\eta(s),-v(s))+g(\eta(s)l(s) +D\phi(\eta(s),\bar t-s)\cdot\dot\eta(s)
-\phi_t(\eta(s),\bar t-s)\big)\d s\cr
\geq &\, \int_0^\tau\big(-v(s)\cdot D\phi(\eta(s),\bar t-s)
-H(\eta(s),D\phi(s,\bar t-s))-g(\eta(s))l(s)\cr
&\, +D\phi(\eta(s),\bar t-s)\cdot\dot\eta(s)
-\phi_t(\eta(s),\bar t-s)\big)\d s>0, 
\end{align*}
which is a contradiction.

It remains to show that $w$ is continuous on $Q$. According to (\ref{eq:cn2}), we have 
$w^*(\cdot,0)=w_*(\cdot,0)=u_0$ on $\ol\gO$. Thus,  
applying the comparison theorem (Theorem 
\ref{thm:c4} with $U:=\R^n$) , we see that $w^*\leq w_*$ on $\ol Q$, which guarantees that 
$w\in C(Q)$.    
This completes the proof. 
\end{proof}   

For the proof of Lemma \ref{thm:cn4}, we need the following basic lemma.

\begin{lem}\label{Lemma A}
Let $R>0$. There is a constant $C>0$, depending only on $R$ and $H$, 
such that for any $(x,p,v)\in\ol\gO\tim B(0,\,R)\tim\R^n$, if 
\[
H(x,p)+L(x,-v)\leq 1-v\cdot p,
\] 
we have $|v|\leq C$. 
\end{lem}

\begin{proof} We may choose a constant $C_1>0$ so that 
\[
C_1\geq \max_{\ol\gO\tim B(0,\,2R)}|H|. 
\]
Observe that 
\[
L(x,-v)\geq \max_{p\in B(0,\,2R)}(-v\cdot p)-C_1 =2R|v|-C_1 
\ \ \hb{ for }(x,\,v)\in\ol\gO\tim\R^n.
\]
Let $(x,p,v)\in\ol\gO\tim B(0,\,R)\tim\R^n$ satisfy  
\[
H(x,p)+L(x,-v)\leq 1-v\cdot p. 
\] 
Then we have 
\[
-C_1+2R|v|-C_1\leq 1+|v||p|\leq 1+R|v|.
\]
Consequently, we get
\[
R|v|\leq 2C_1+1.\qedhere
\]
\end{proof}

For $i\in\N$ we introduce the function $L_i\in C(\ol\gO\tim\R^n)$ by setting  
\[
L_i(x,\xi)=\max_{p\in B(0,i)}\big(\xi\cdot p-H(x,p)\big). 
\]
Observe that $L_i(x,\xi)\leq L(x,\xi)$ and 
$\lim_{i\to\infty}L_i(x,\xi)=L(x,\xi)$ for $(x,\xi)\in\ol\gO\tim\R^n$ and  
that every $L_i$ is uniformly continuous on bounded subsets of $\ol\gO\tim\R^n$.

\begin{proof}[Proof of Lemma {\em\ref{thm:cn4}}]  
Fix $k\in\N$. Set $\gd=t/k$ and $s_j=(j-1)\gd$ for $j=1,2,...,k$. 
We define inductively a sequence 
$\{(x_j,\eta_j,v_j,l_j)\}_{j=1}^k\subset \ol\gO\tim \SP$. 
We set $x_1=x$ and choose a $\xi_1\in\R^n$ so that  
\[
H(x_1,D\phi(x_1,t))+L(x_1,-\xi_1)\leq \gep-\xi_1\cdot D\phi(x_1,t).
\]
Set $v_1(s)=\xi_1$ for $s\geq 0$ and choose a pair $(\eta_1,l_1)\in \Lip(\ol\R_+,\,\ol\gO)\tim 
L^\infty(\ol\R_+,\,\ol\R_+)$ so that 
$(\eta_1,v_1,l_1)\in\SP(x_1)$. 
According to Theorem \ref{thm:s1}, such a pair always exists. 

Suppose now that we are given 
$(x_i,\eta_i,v_i,l_i)$ for all $i=1,2,...,j-1$ and for some $j\leq k$. 
Then set $x_j=\eta_{j-1}(\gd)$, 
choose a $\xi_j\in\R^n$ so that 
\beeq
H(x_j,D\phi(x_j,t-s_j))+L(x_j,-\xi_j)\leq \gep-\xi_j\cdot D\phi(x_j,t-s_j), 
\label{eq:cn7}
\eneq
set $v_j(s)=\xi_j$ for $s\geq 0$, and select a pair $(\eta_j,l_j)\in \Lip(\ol\R_+,\ol\gO)\tim 
L^\infty(\ol\R_+,\R^n)$ so that 
$(\eta_j,v_j,l_j)\in\SP(x_j)$. 
Thus, by induction, we have chosen a sequence 
$\{(x_j,\eta_j,v_j,l_j)\}_{j=1}^k\subset \ol\gO\tim \SP$ such that 
$x_1=\eta_1(0)$, $x_j=\eta_{j-1}(\gd)=\eta_j(0)$ for $j=2,...,k$ 
and for each $j=1,2,...,k$, 
(\ref{eq:cn7}) holds with $\xi_j=v_j(s)$ for all $s\geq 0$.  
Notice that the choice of  
$x_j,\,\eta_j,\,v_j,\,l_j$, with $j=1,...,k$, depends on $k$, 
which is not explicit 
in our notation.  

Next, we define a triple $(\bar\eta_k,\bar v_{k},\bar l_k)\in \SP(x)$ by 
setting
\[\big(\bar \eta_k(s),\,\bar v_k(s),\,\bar l_k(s)\big)
=\big(\eta_j(s-s_j),\, v_j(s-s_j),\,l_j(s-s_j)\big)
\] 
for $s_j\leq s<s_{j+1}$ and $j=1,2,...,k-1$ 
and  
\[\big(\bar\eta_k(s),\,\bar v_k(s),\,\bar l_k(s)\big)
=\big(\eta_k(s-s_k),\, v_k(s-s_k),\,l_k(s-s_k)\big) 
\] 
for $s\geq s_k$. 
We may assume that $\gep<1$ and, by Lemma \ref{Lemma A}, we find a constant $C_1>0$, 
independent of 
$k$, such that
$\max_{s\geq 0}|\bar v_k(s)|=\max_{1\leq j\leq k}|\xi_j|\leq C_1$. 
By Proposition 
\ref{thm:s2}, we find a constant $C_2>0$, independent of $k$, such that 
$\|\dot\eta_k\|_{L^\infty(\R_+)}\vee \|\bar l_k\|_{L^\infty(\R_+)}\leq C_2$. 
Now, we define the step function $\chi_k$ on $\ol\R_+$ by  
setting $\chi_k(s)=s_j$ for $s_j\leq s<s_{j+1}$  
and $j=1,2,...,k$ and $\chi_k(s)=s_k$ for $s\geq s_k$, and   
observe that (\ref{eq:cn7}), $1\leq j\leq k$, can be rewritten as
\begin{equation}\label{eq:cn8}
\begin{split} 
& H(\bar \eta_k(\chi_k(s)), D\phi(\bar\eta_k(\chi_k(s)),t-\chi_k(s)))
+L(\bar\eta_k(\chi_k(s)),-\bar v_k(s))\\
&\qquad\qquad \leq \gep-\bar v_k(s)\cdot D\phi(\bar\eta_k(\chi_k(s)), t-\chi_k(s)) \ \ \hb{ for }0\leq s\leq t. 
\end{split}
\end{equation}

We may invoke Proposition \ref{thm:s3}, to find a triple $(\eta,\,v,\,l)\in\SP(x)$ 
and a subsequence of $\{(\bar \eta_k,\bar v_k, \bar l_k)\}_{k\in\N}$, 
which will be denoted again by the same symbol, so that for every 
$0<T<\infty$, as $k\to \infty$,  
$\bar\eta_k \to \eta$ uniformly on $[0,\,T]$, 
$\bar v_k\d s \to v\d s$ weakly-star in $C([0,\,T],\,\R^n)^*$ and 
$\bar l_k\d s\to l\d s$  weakly-star in $C([0,\,T])^*$. We may moreover 
assume that $\bar v_k \to v$ weakly-star in $L^\infty(\R_+,\,\R^n)$ and 
$\bar l_k\to l$  weakly-star in $L^\infty(\R_+)$ as $k\to\infty$.    

Since $\bar v_k\to v$ weakly in $L^2(0,t)$, we may 
choose a sequence $\{\gl_k\}_{k\in\N}$ of finite sequences 
$\gl_k=(\gl_{k,1},\gl_{k,2},...,\gl_{k,N_k})$ of nonnegative 
numbers such that 
\[
\sum_{j=1}^{N_k}\gl_{k,j}=1 \ \ \hb{ and } \ \ 
\hat v_k:=\sum_{j=1}^{N_k}\gl_{k,j}v_j \ \hb{ converge to } \ v \ \hb{ in } L^2(0,\,t). 
\]
Here we may moreover assume by selecting a subsequence of $\{\bar\eta_k,\bar v_k,\bar l_k)\}$ 
that as $k\to\infty$, $\hat v_k(s) \to v(s)$ for a.e. $s\in (0,\,t)$.   

Fix any $i\in\N$ and $\gth>1$. In view of the uniform 
continuity of the functions $H$ and $L_i$ on bounded subsets of 
$\ol\gO\tim\R^n$ and the uniform convergence of $\{\bar \eta_k\}$ 
to $\eta$ on $[0,\,t]$, 
from (\ref{eq:cn8}), we get 
\begin{align*}
&H(\eta(s),D\phi(\eta(s),t-s))+L_i(\eta(s),-\bar v_k(s)) \\
&\leq \gth \gep -\bar v_k(s)\cdot D\phi(\eta(s),t-s) \ \ \hb{ for }s\in(0,\,t)
\end{align*}
for sufficiently large $k$, say, for $k\geq k_\gth$, 
and hence, by taking the convex combination,  
\begin{align*}
&H(\eta(s),D\phi(\eta(s),t-s))+L_i(\eta(s),-\hat v_k(s)) \\
&\leq \gth \gep -\hat v_k(s)\cdot D\phi(\eta(s),t-s) \ \ \hb{ for }s\in(0,\,t)
\end{align*}
for $k\geq k_\gth$. Sending $k\to\infty$, we get 
\[
H(\eta(s),D\phi(\eta(s),t-s))+L_i(\eta(s),-v(s))
\leq \gth\gep -v(s)\cdot D\phi(\eta(s),t-s) \ \ \hb{ for a.e. } s\in(0,\,t),
\]
and, because of the arbitrariness of $i$ and $\gth>1$, 
we obtain 
\[
H(\eta(s),D\phi(\eta(s),t-s))+L(\eta(s),-v(s))
\leq \gep -v(s)\cdot D\phi(\eta(s),t-s) \ \ \hb{ for a.e. } s\in(0,\,t).\qedhere
\]
\end{proof}

\section{Aubry-Mather sets and formulas for solutions of (\ref{eq:i1}), (\ref{eq:i2})}  

In this section we define the Aubry-Mather set associated with (\ref{eq:i1}), (\ref{eq:i2}). 
Our argument here is very close to that of \cite{FathiSiconolfi_04,  
FathiSiconolfi_05}. 

By the $C^1$ regularity of $\gO$ and assumption (A3), there is 
a function $\psi\in C^1(\ol\gO)$ such that $D_\gg\psi(x)>0$ for $x\in \gG$. 
By multiplying $\psi$ by a positive constant, we may assume that $D_\gg\psi(x)\geq |g(x)|$ 
for $x\in\gG$. Selecting a constant $C_-\in\R$ small enough, we may have 
$H(x,D\psi(x))\geq C_-$ for $x\in\gO$.  It is easy to check that 
the function $\psi$ is a supersolution of (\ref{eq:i1}), (\ref{eq:i2}), with $C_-$ in 
place of $a$. Similarly, if we choose a constant $C_+\in\R$ large enough, then 
the function $-\psi$ is a subsolution of (\ref{eq:i1}), (\ref{eq:i2}), 
with $C_+$ in place of $a$. 

We define the {\em critical value} (or additive eigenvalue) $c$ by
\[
c=\inf\{a\in\R\mid \hb{there is a subsolution of (\ref{eq:i1}), (\ref{eq:i2})}\}. 
\]
Obviously we have $c\leq C_+$. By Corollary \ref{thm:c2}, we see as well that 
$c\geq C_-$. In particular, we have $c\in\R$.  For any decreasing sequence 
$\{a_k\}$ converging to $c$, there is a sequence $\{u_k\}\subset \USC(\ol\gO)$ 
such that for every $k\in\N$, $u_k$ is a subsolution of (\ref{eq:i1}), (\ref{eq:i2}), 
with $a_k$ in place of $a$. By Lemma \ref{thm:c3}, with $U=\R^n$, we find that 
$\{u_k\}$ is equi-Lipschitz continuous on $\ol\gO$. By adding a constant to $u_k$, 
we may assume that $\{u_k\}$ is uniformly bounded on $\ol\gO$. 
By choosing a subsequence, we may thus assume that the sequence $\{u_k\}$ converges 
to a function $u\in\Lip(\ol\gO)$ as $k\to\infty$. By the stability 
of the viscosity property under uniform 
convergence, we see that $u$ is a subsolution of (\ref{eq:i1}), (\ref{eq:i2}), with $c$ 
in place of $a$. 

Henceforth in this section, we normalize $c=0$ by replacing $H$ by $H-c$, 
and we are concerned only with problem (\ref{eq:i1}), (\ref{eq:i2}), 
with $a=0$, that is, the problem  
\begin{equation}\label{eq:am1}
\left\{
\begin{aligned}
&H(x,\, Du(x))=0 \ \ \hb{ in }\gO,\\
&D_\gg u(x)=g(x) \ \ \hb{ on }\gG.
\end{aligned}
\right.
\end{equation}
We introduce the function $d$ on $\ol\gO\tim\ol\gO$ by
\begin{equation}\label{eq:def-d}
d(x,y)=\sup\{v(x)-v(y)\mid v \hb{ is a subsolution of (\ref{eq:am1})}\}. 
\end{equation}
According to Lemma \ref{thm:c3}, the family of functions $d(\cdot,y)$, with $y\in\ol\gO$, 
is equi-Lipschitz continuous on $\ol\gO$. 
By the stability of the viscosity property, we see that for any $y\in\ol\gO$, the function 
$d(\cdot,y)$ is a subsolution of (\ref{eq:am1}). It is easily seen that 
\[
d(x,y)\leq d(x,z)+d(z,y) \ \ \hb{ for }x,y,z\in\ol\gO. 
\]
Also, in view of the Perron method, 
we find that for every $y\in\ol\gO$, the function $d(\cdot,y)$ is a solution of 
\begin{equation}\label{eq:am2}
\left\{
\begin{aligned}
&H(x,\, Du(x))=0 \ \ \hb{ in }\gO\setminus\{y\},\\
&D_\gg u(x)=g(x) \ \ \hb{ on }\gG\setminus\{y\},
\end{aligned}
\right.
\end{equation}
which is just problem (\ref{eq:b1}), with $f:=0$ and $U:=\R^n\setminus\{y\}$.

We define the {\em Aubry-Mather set} $\cA$ associated with (\ref{eq:am1}) 
(or (\ref{eq:i1}), (\ref{eq:i2}) with generic $a$) by
\[
\cA=\{y\in\ol\gO\mid d(\cdot,y)\hb{ is a solution of (\ref{eq:am1})}\}. 
\] 

\begin{thm}\label{thm:am1}
The Aubry-Mather set $\cA$ is a nonempty and compact. 
\end{thm}

\begin{remark}If we define the function $d_a$ on $\ol\gO\tim\ol\gO$ by
\[
d_a(x,y)=\sup\{v(x)-v(y)\mid v\hb{ is a subsolution of (\ref{eq:i1}), (\ref{eq:i2})}\}, 
\] 
then $d_a(x,y)=\sup\emptyset=-\infty$ for $a<0$. Moreover, if we define the Aubry-Mather set 
$\cA_a$ for $a>0$ by 
\[
\cA_a=\{y\in\ol\gO\mid d_a(\cdot,y)\hb{ is a solution of (\ref{eq:i1}), (\ref{eq:i2})}\},
\]
then $\cA_a=\emptyset$. 

\end{remark}

The non-emptiness of $\cA$ will be proved based on the following observation. 

\begin{lem}\label{thm:am2}
Let $y\in\ol\gO\setminus\cA$. Then there are functions $v\in \Lip(\ol\gO)$ and $f\in C(\ol\gO)$ 
such that $f(y)<0$, $f(x)\leq 0$ for $x\in\ol\gO$ and $v$ is a subsolution of 
{\em(\ref{eq:b1})}, 
with $U=\R^n$. 
\end{lem}
    
\begin{proof}
Fix any $y\in\ol\gO\setminus \cA$ and set $u(x)=d(x,y)$ for $x\in\ol\gO$. 
For definiteness, we consider the case where $y\in\gG$. We leave it to the reader 
to check the other case. Since $u$ is not a supersolution of (\ref{eq:am1}) while it is a 
solution of (\ref{eq:am2}),
we find a $C^1$ function $\phi$ on $\ol\gO$ such that $u-\phi$ attains a 
strict minimum at $y$,
\[
H(y,\,D\phi(y))<0 \ \ \hb{ and } \ \ D_\gg\phi(y)<g(y). 
\]
By continuity, there is an open neighborhood $V$ of $y$ such that
\begin{equation}\label{eq:am3}
H(x,\,D\phi(x))<0 \ \hb{ for } x\in \gO_V  \ \ 
\hb{ and } \ \ D_\gg\phi(x)<g(x) \ \hb{ for }x\in\gG_V. 
\end{equation}
We may assume by adding a constant to $\phi$ that $u(y)=\phi(y)$. Note that
$\min_{\ol\gO\setminus V} (u-\phi)>0$, and 
select a constant $\gep>0$ small enough so that $(u-\phi)(x)>\gep$ for $x\in\ol\gO\setminus V$. 
We may choose an open neighborhood $W$ of $V^c$ such that $(u-\phi)(x)>\gep$ for $x\in \ol\gO\cap W$.  
We set $v(x)=u(x)\vee (\phi(x)+\gep)$ for $x\in\ol\gO$. 

Observe that $v(x)=u(x)$ for $x\in W\cap \ol\gO$, which ensures that 
$v$ is a subsolution of (\ref{eq:b1}), with $f(x):=0$ and $U:=W$. 
On the other hand, 
there is an open neighborhood $Y\subset V$ of $y$ such that 
$\phi(x)+\gep>u(x)$ for $x\in Y\cap\ol\gO$.  It is clear that 
$\ol\gO\cap Y\cap W=\emptyset$.  In view of (\ref{eq:am3}), 
we may choose a function $f\in C(\ol\gO)$ so that $f(y)<0$, $f(x)\leq 0$ for $x\in Y$, 
$f(x)=0$ for $x\in\ol\gO\setminus Y$ and 
\[
H(x,\,D\phi(x))\leq f(x) \ \hb{ for } x\in \gO_V \ \ 
\hb{ and } \ \ D_\gg\phi(x)\leq g(x) \ \hb{ for }x\in\gG_V. 
\]
It is easily seen that $v$ is a subsolution of (\ref{eq:b1}), with $U:=V$. 
Finally, we note that $v$ is a subsolution of (\ref{eq:b1}), with  
$U:=\R^n$, and finish the proof. 
\end{proof}

\begin{proof}[Proof of Theorem {\em \ref{thm:am1}}]
The compactness of $\cA$ follows directly from the stability 
of the viscosity property under uniform convergence. 

To see that $\cA\not=\emptyset$, we suppose by contradiction that 
$\cA=\emptyset$. By Lemma \ref{thm:am2}, for each $y\in\ol\gO$ there are 
functions $v_y\in\Lip(\ol\gO)$ and $f_y\in C(\ol\gO)$ such that 
$f_y(y)<0$, $f_y(x)\leq 0 $ for $x\in\ol\gO$ and  $v_y$ is a subsolution of 
(\ref{eq:b1}), with $f:=f_y$ and $U:=\R^n$. By the compactness of $\ol\gO$, 
we may choose a finite sequence $\{y_j\}_{j=1}^J\subset\ol\gO$ so that 
$\sum_{j=1}^J f_{y_j}(x)<0$ for $x\in\ol\gO$. Theorem 
\ref{thm:b2}, with $U:=\R^n$, guarantees that the function 
\[
v(x)=\fr{1}{J}\sum_{j=1}^Jv_{y_j}(x)
\]
on $\ol\gO$ is a subsolution of (\ref{eq:b1}), with $U:=\R^n$ and
\[
f(x):=\fr{1}{J}\sum_{j=1}^J f_{y_j}(x). 
\]
We choose a constant $a<0$ so that $f(x)\leq a$ for $x\in\ol\gO$ and observe that 
$v$ is a subsolution of (\ref{eq:b1}), with $f:=a$ and $U:=\R^n$. This contradicts 
the fact that $c=0$. The proof is complete.  
\end{proof}

\begin{prop}\label{thm:am3} The function $d$ can be represented as 
\begin{equation}\label{eq:am4}
\begin{split}
d(x,y)=\inf\Big\{\int_0^t \big(L(\eta(s),\,-v(s))+&\,g(\gg(s))l(s)\big)\d s\mid t>0,\\ 
&(\eta,\,v,\,l)\in \SP(x),\ \eta(t)=y\Big\}. 
\end{split}
\end{equation}
\end{prop}

\begin{proof} Fix any $y\in\ol\gO$. We denote by $w(x)$ 
the right side of (\ref{eq:am4}). 
According to Theorem \ref{thm:cn1}, the function 
\[
\begin{split}
u(x,t):=\inf\Big\{\int_0^t & L(\eta(s),\,-v(s))+g(\eta(s))l(s)\d s \\
&+d(\eta(t),y)\mid (\eta,\,v,\,l)\in\SP(x)\Big\}
\end{split}
\]
is a solution of (\ref{eq:i3})--(\ref{eq:i5}), with $u_0:=d(\cdot,\,y)$. 
Noting that the function 
$d(x,y)$, as a function of $(x,t)\in\ol\gO\tim \ol\R_+$, is 
a subsolution of (\ref{eq:i3})--(\ref{eq:i5}) with $u_0:=d(\cdot,\,y)$,  
by applying the comparison theorem (Theorem \ref{thm:c4}, 
with $U=\R^n$), we see that $d(x,y)\leq u(x,t)$ for $(x,t)\in\ol\gO\tim \R_+$. 
Since $d(y,y)=0$, we have $\inf_{t>0}u(x,t)\leq w(x)$ for $x\in\ol\gO$. 
Consequently, we have $d(x,y)\leq w(x)$ for $x\in\ol\gO$.  

By the $C^1$ regularity of $\gO$, for each $x\in\ol\gO$ we may choose a 
Lipschitz continuous curve $\eta$ on $[0,\,t]$ connecting $x$ to $y$ in $\ol\gO$,
with a Lipschitz constant independent of $x$. Here $t>0$ is an appropriate constant, 
and moreover we may assume that 
$t\leq C_1|x-y|$ for some constant $C_1>0$ independent of $x$.   
As is well-known and easily shown, $L(x,\xi)$ is bounded on $\ol\gO\tim B(0,\gd)$, if $\gd>0$ is chosen 
sufficiently small. 
Fix such a constant $\gd>0$ and choose a constant $C_2>0$ so that 
$L(x,\xi)\leq C_2$ for $(x,\xi)\in\ol\gO\tim B(0,\gd)$.  
By scaling, we may assume that $|\dot\eta(s)|\leq \gd$ for a.e. $s\in[0,\,t]$.  
Noting that $(\eta,\,\dot\eta,\,0)\in\SP(x)$, we get 
\[
w(x)\leq \int_0^t L(\eta(s),-\dot\eta(s))\d s \leq C_2 t\leq C_1C_2|x-y|.
\]
In particular, we may conclude that $w$ is continuous at $y$ and $w(y)=0$.

To complete the proof, it is enough to show that $w$ is a subsolution of (\ref{eq:am1}). Indeed, once this is done, 
by the definition of $d$, we get
\[
w(x)=w(x)-w(y)\leq d(x,y) \ \ \hb{ for }x\in\ol\gO,
\]
which guarantees that $d(x,y)=w(x)$ for $x\in\ol\gO$. 

To prove the subsolution property of $w$, we just need to follow the argument of the proof 
of Theorem \ref{thm:cn1}. 
Let $\x\in \ol\gO$ and $\phi\in C^1(\ol \gO)$. Assume that 
$w^*-\phi$ attains a strict maximum at $\x$. We need to show that if $\x\in\gO$, then  
$H(\x,D\phi(\x))\leq 0$, and if $\x\in\gG$, then either 
\begin{equation}
H(\x,D\phi(\x))\leq 0 \ \ \hb{ or } \ \ \gg(\x)\cdot D\phi(\x)\leq g(\x). 
\label{eq:am5}
\end{equation}
We are here concerned only with the case where $\x\in\gG$, and leave the proof 
in the other case to the reader.  
To show (\ref{eq:am5}), we suppose by 
contradiction that (\ref{eq:am5}) were false. 
Then we may choose an $\gep\in(0,\,1)$ so that 
for $x\in \ol\gO\cap B(\x,\,2\gep)$, 
\beeq
H(x,D\phi(x))\geq 2\gep \ \ \hb{ and  } \ \ 
\gg(x)\cdot D\phi(x)-g(x)\geq 2\gep,
\label{eq:am6}
\eneq  
where $\gg$ and $g$ are, as usual, assumed to be defined and continuous on $\ol\gO$. 
We may also assume that $(w^*-\phi)(\x)=0$. 
Set
\[
B=\pl B(\x,2\gep)\, \cap\, \ol \gO,
\]
and 
$m=-\max_{B}(w^*-\phi)$.  
Obviously, we have $m>0$ and
$w(x)\leq \phi(x)-m$ for $x\in B$. 
We choose a point $\bar x\in \ol\gO\cap B(\x,\gep)$ 
so that $(w-\phi)(\bar x)>-\gep^2\wedge m$.  
We apply Lemma \ref{thm:cn4}, to obtain a triple $(\eta,v,l)\in\SP(\bar x)$ such that 
for a.e. $s\geq 0$, 
\beeq
H(\eta(s),D\phi(\eta(s)))+L(\eta(s),-v(s))\leq \gep-v(s)\cdot D\phi(\eta(s))
\label{eq:am7}
\eneq 
Note that $\dist(\bar x, \pl B(\x,2\gep))\geq \gep$, and set 
\[
\tau=\inf\{s>0\mid \eta(s)\in \pl B(\x,2\gep)\}. 
\]
Consider first the case where $\tau=\infty$, which means that 
$\eta(s)\in\Int B(\x,2\gep)$ for all $s\geq 0$. 
By the dynamic programming 
principle, we have
\[
\phi(\bar x)< w(\bar x)+\gep^2
\leq  \int_0^\gep \big(L(\eta(s),-v(s))+g(\eta(s))l(s)+\gep\big)\d s
+\phi(\eta(\gs)).
\]
Hence, we obtain 
\begin{align*}
0<&\, \int_0^\gep\big(L(\eta(s),-v(s))+g(\eta(s))l(s)+\gep
+D\phi(\eta(s))\cdot\dot\eta(s)\big)\d s\\ 
\leq &\, \int_0^\gep\big\{L(\eta(s),-v(s))+g(\eta(s))l(s)+\gep
+D\phi(\eta(s))\cdot\big(v(s)-l(s)\gg(\eta(s))\big)
\big\}\d s.
\end{align*}
Now, using (\ref{eq:am7}) and (\ref{eq:am6}), we get 
\begin{align*}
0&<\int_0^\gep\big\{2\gep-H(\eta(s),D\phi(\eta(s)))+g(\eta(s))l(s)
-D\phi(\eta(s))\cdot \gg(\eta(s))l(s)\big\}\d s \\
&\leq 0, 
\end{align*}
which is a contradiction. 

Consider next the case where $\tau<\infty$. Note that 
\begin{align*}
\phi(\bar x)<&\, w(\bar x)+m 
\leq \int_0^\tau 
\big(L(\eta(s),-v(s))+g(\eta(s))l(s)\big)\d s+w(\eta(\tau))+m\\
\leq&\,  \int_0^\tau 
\big(L(\eta(s),-v(s))+g(\eta(s))l(s)\big)\d s+\phi(\eta(\tau)).
\end{align*}
Using (\ref{eq:am7}) and (\ref{eq:am6}) as before, we obtain 
\[
0<\int_0^\tau \big\{\gep-H(\eta(s),D\phi(\eta(s)))
+l(s)[g(\eta(s))-\gg(\eta(s))\cdot D\phi(\eta(s))]\big\}\d s
<0. 
\]
This is again a contradiction, and we conclude that $w$ 
is a subsolution of (\ref{eq:am1}).  
\end{proof}

We give another characterization of the Aubry-Mather set associated with (\ref{eq:am1}). 

\begin{thm}\label{thm:am4}Let $\tau>0$ and $y\in\ol\gO$. 
Then we have $y\in\cA$ if and only if 
\begin{equation}\label{eq:am8}
\begin{split}
\inf\Big\{\int_0^t \big(L(\eta(s),\,-v(s))+&\,g(\eta(s))l(s)\big)\d s\mid t>\tau,\\
& (\eta,\,v,\,l)\in\SP,\
 \eta(0)=\eta(t)=y\Big\}=0.
 \end{split} 
\end{equation}
\end{thm}

\begin{lem}\label{thm:am5} Let $u_0\in C(\ol\gO)$ and let 
$u\in C(\ol\gO\tim\ol\R_+)$ be the solution of {\em(\ref{eq:i3})--(\ref{eq:i5})}, with $a:=0$.    
Set 
\[
u^-(x)=\liminf_{t\to\infty}u(x,t) \ \ \hb{ for }x\in\ol\gO.
\]
Then $u^-\in\Lip(\ol\gO)$ and it is a solution of {\em(\ref{eq:am1})}.  
\end{lem}

\begin{proof} Thanks to Theorem \ref{thm:am1}, there is a solution $\phi\in\Lip(\ol\gO)$ 
of (\ref{eq:am1}). By adding a constant to $\phi$ if needed, we may assume that 
$\phi(x)\leq u_0(x)$ for $x\in\ol\gO$. Let $C>0$ be a constant such that 
$u_0(x)\leq \phi(x)+C$ for $x\in\ol\gO$. By comparison, we get 
$\phi(x)\leq u(x,t)\leq \phi(x)+C$ for $x\in\ol\gO$. 

Setting $v(x,t)=\inf_{s>t}u(x,s)$ for $(x,t)\in\ol\gO\tim\R_+$,  we note that 
\[
u^-(x)=\sup_{t>0}v(x,t) \ \ \hb{ for }x\in\ol\gO.  
\]
Applying Theorem \ref{thm:b8} (and the remark after 
it) to the family $\{u(\cdot,\cdot+s)\}_{s>0}$ of solutions 
of (\ref{eq:i3}), (\ref{eq:i4}), with $a:=0$, 
we see that $v$ is a solution of (\ref{eq:i3}), (\ref{eq:i4}), with $a:=0$. 
Observe also that $v\in \USC(\ol\gO\tim\R_+)$ and the functions $v(x,\cdot)$, 
with $x\in\ol\gO$, are nondecreasing on $\R_+$. This monotonicity of $v$ guarantees that 
the functions $v(\cdot,t)$, with $t>0$, are subsolution of (\ref{eq:am1}), 
which implies that the family $\{v(\cdot,t)\}_{t>0}$ is equi-Lipschitz 
continuous on $\ol\gO$. Accordingly, we have $u^-\in\Lip(\ol\gO)$. 
By the Dini lemma, we see that     
\[
u^-(x)=\lim_{t\to\infty}v(x,t) \ \ \hb{ uniformly for }x\in\ol\gO. 
\] 
By the stability of viscosity property under uniform convergence,   
we conclude that $u^-$ is a solution of (\ref{eq:am1}). 
\end{proof}

\begin{proof}[Proof of Theorem {\em\ref{thm:am4}}] Fix any $\tau>0$ and $y\in\ol\gO$. By Proposition \ref{thm:am3}, 
we have
\begin{align}\label{eq:am9}
\inf\Big\{\int_0^t \big(L(\eta(s),-v(s))+\,&g(\eta(s))l(s)\big)\d s\mid (\eta,\,v,\,l)
\in \SP,\ \\
&\eta(0)=\eta(t)=y\Big\}
\geq d(y,y)=0 \ \ \hb{ for }t>0. \nonumber 
\end{align}

We assume that $y\in\cA$ and 
show that (\ref{eq:am8}) holds. 
Note that the function $u(x,t)=d(x,y)$ on $\ol\gO\tim\R$ is the unique solution of 
the initial-boundary value problem (\ref{eq:i3})--(\ref{eq:i5}), 
with $u_0:=d(\cdot,\,y)$. By Theorem \ref{thm:cn1}, we 
get 
\begin{align*}
0=\,&d(y,y)\\
=\,&\inf\Big\{\int_0^\tau \big(L(\eta(s),-v(s))+g(\eta(s))l(s)\big)\d s +
d(\eta(\tau),y)\mid (\eta,\,v,\,l)\in\SP(y)\Big\}. 
\end{align*}
Fix any $\gep>0$ and choose a triple $(\eta,\,v,\,l)\in\SP(y)$ 
so that 
\[
\gep>\Big\{\int_0^\tau \big(L(\eta(s),-v(s))+g(\eta(s))l(s)\big)\d s +
d(\eta(\tau),y).
\]
In view of Proposition \ref{thm:am3}, by modifying $(\eta,\,v,\,l)$ 
on the set $(\tau,\,\infty)$ if necessary, we may assume that for some $t>\tau$, 
\[
d(\eta(\tau),y)+\gep>\int_\tau^t \big(L(\eta(s),\,-v(s))+g(\eta(s))l(s)\big)\d s
\ \ \hb{ and } \ \ \eta(t)=y. 
\]
Thus, we obtain
\[
2\gep>\int_0^t \big(L(\eta(s),\,-v(s))+g(\eta(s))l(s)\big)\d s \ \ \hb{ and } \ \ 
\eta(0)=\eta(t)=y, 
\]
which ensures together with (\ref{eq:am9}) that (\ref{eq:am8}) holds. 

Next we assume that (\ref{eq:am8}) holds and show that $y\in\cA$. 
Let $u$ be the unique solution of problem (\ref{eq:i3})--(\ref{eq:i5}), 
with initial data $d(\cdot,y)$. Since $d(\cdot,y)$, regarded as a function 
on $\ol\gO\tim\ol\R_+$, is a subsolution of 
(\ref{eq:i3}), (\ref{eq:i4}), by comparison, we see that 
$d(x,y)\leq u(x,t)$ for $(x,t)\in\ol\gO\tim[0,\,\infty)$. 
As in Lemma \ref{thm:am5}, we set 
\[
u^-(x)=\liminf_{t\to\infty}u(x,t) \ \ \hb{ for }x\in\ol\gO,  
\]
to find that $u^-\in\Lip(\ol\gO)$ and $u^-$ is a solution of (\ref{eq:am1}).  
It follows that $d(x,y)\leq u^-(x)$ for $x\in\ol\gO$.  
It is easily seen from (\ref{eq:am8}) that for each $k\in\N$,
\[
\begin{split}
\inf\Big\{\int_0^t \big(L(\eta(s),\,v(s))+&\,g(\eta(s))l(s)\big)\d s\mid t>k\tau,\\ 
&(\eta,\,v,\,l)\in\SP,\ \eta(0)=\eta(t)=y\Big\}=0.
\end{split}
\]
On the other hand, we have 
\begin{align*}
\inf_{t>k\tau} u(y,t)\leq 
\inf\Big\{\int_0^t \big(L(\eta(s),\,v(s))+&\,g(\eta(s))l(s)\big)\d s\mid t>k\tau,\\ 
&(\eta,\,v,\,l)\in\SP,\ \eta(0)=\eta(t)=y\Big\}.
\end{align*}
These together ensure that $u^-(y)\leq 0$ and hence $d(x,y)\geq u^-(x)$ for $x\in\ol\gO$. 
Thus we find that $d(x,y)=u^-(x)$ and conclude that $y\in\cA$. 
\end{proof}

\begin{thm}\label{thm:am6}Let $u\in\USC(\ol\gO)$ and $v\in\LSC(\ol\gO)$ be 
respectively a subsolution and a supersolution of {\em(\ref{eq:am1})}. 
Assume that $u(x)\leq v(x)$ for $x\in\cA$. Then $u(x)\leq v(x)$ for $x\in\ol\gO$. 
\end{thm}

\begin{lem}\label{thm:am7}There exist functions $\psi\in\Lip(\ol\gO)$ and 
$f\in C(\ol\gO)$ such that $f(x)\leq 0$ for $x\in\ol\gO$, 
$f(x)<0$ for $x\in\ol\gO\setminus\cA$ and $\psi$ is a subsolution of 
{\em (\ref{eq:b1})}, with $U:=\R^n$.  
\end{lem}

\begin{proof} By Lemma \ref{thm:am2}, for each $y\in\ol\gO\setminus\cA$ there 
are functions $f_y\in C(\ol\gO)$ and $\psi_y\in C(\ol\gO)$ such that 
$f_y(y)<0$, $f_y(x)\leq 0$ for $x\in\ol\gO$ and $\psi_y$ is a subsolution of 
(\ref{eq:b1}), with $U:=\R^n$ and $f:=f_y$. Since 
$\{\psi_y\}_{y\in\ol\gO\setminus\cA}$ is equi-Lipschitz continuous 
on $\ol\gO$, we may assume by adding to $\psi_y$ an appropriate 
constant $C_y\in\R$ if necessary   
that $\{\psi_y\}_{y\in\ol\gO\setminus\cA}$ is uniformly 
bounded on $\ol\gO$. Also, we may assume without any loss of 
generality that 
$\{f_y\}_{y\in\ol\gO\setminus\cA}$ is uniformly bounded on $\ol\gO$.  
We may choose a sequence $\{y_j\}_{j\in\N}\subset \ol\gO\setminus\cA$ 
so that 
\[
\inf_{j\in\N}f_{y_j}(x) <0 \ \ \hb{ for }x\in \ol\gO\setminus \cA. 
\]
Now we set 
\[
\psi(x)=\sum_{j\in\N}2^{-j}\psi_{y_j}(x) \ \ \hb{ for }x\in\ol\gO,
\]
and observe in view of Theorem \ref{thm:b2} that 
$\psi$ is a subsolution of (\ref{eq:b1}), with $U:=\R^n$ 
and $f$ given by 
\[
f(x)=\sum_{j\in\N}2^{-j}f_{y_j}(x) \ \ \hb{ for }x\in\ol\gO. 
\] 
Finally, we note that $f(x)\leq 0$ for $x\in\ol\gO$, 
$f(x)<0$ for $x\in\ol\gO\setminus \cA$ and $\psi\in\Lip(\ol\gO)$. 
The proof is complete. 
\end{proof}

\begin{proof}[Proof of Theorem {\em\ref{thm:am6}}]
Due to Lemma \ref{thm:am7}, there are functions 
$f\in C(\ol\gO)$ and $\psi\in\Lip(\ol\gO)$ such that 
$f(x)\leq 0$ for $x\in\ol\gO$, $f(x)<0$ for $x\in\ol\gO\setminus\cA$ 
and $\psi$ is a subsolution of (\ref{eq:b1}), with $U:=\R^n$. 
Fix any $0<\gep<1$ and set 
\[
u_\gep(x)=(1-\gep)u(x)+\gep \psi(x) \ \ \hb{ for }x\in\ol\gO. 
\]
Then the function $u_\gep$ is a subsolution 
of (\ref{eq:b1}), with $U:=\R^n$ and $f$ 
replaced by $\gep f$. 
We apply Theorem \ref{thm:c1}, with $U:=\R^n\setminus\cA$, 
to obtain $u_\gep\leq v$ on $\ol\gO$, which implies that 
$u\leq v$ on $\ol\gO$. 
\end{proof}

\begin{thm}\label{thm:am7+} 
Let $u\in C(\ol\gO)$ be a solution of {\em(\ref{eq:am1})}. Then  
\beeq\label{eq:am10}
u(x)=\min\{u(y)+d(x,y)\mid y\in\cA\} \ \ \hb{ for }x\in\ol\gO.
\eneq
\end{thm}

\begin{proof} We denote by $w(x)$ the right hand side of (\ref{eq:am10}). 
We note first by the remark after Theorem \ref{thm:b7} that 
$w$ is a solution of (\ref{eq:am1}).  
Next, by the definition of $d$, we have 
$u(x)-u(y)\leq d(x,y)$ for $x,y\in\ol\gO$. Hence we get 
$u(x)\leq w(x)$ for $x\in\ol\gO$. Also, by the definition of $w$, we have 
$w(x)\leq u(x)$ for $x\in\cA$. Thus we have $u(x)=w(x)$ for $x\in\cA$. By Theorem 
\ref{thm:am6}, we conclude that $u=w$ on $\ol\gO$.  
\end{proof} 

\begin{cor}If $u\in C(\ol\gO)$ is a solution of {\em(\ref{eq:am1})}, then    
\begin{align*}
u(x)=\inf\Big\{\int_0^t \big(L(\eta(s),-v(s))+&\,g(\eta(s))l(s)\big)\d s+u(\eta(t))\mid t>0, \\
&\,(\eta,\,v,\,l)\in\SP(x),\ \eta(t)\in\cA
\Big\} \ \hb{ for }x\in\ol\gO. 
\end{align*}
\end{cor}

Theorem \ref{thm:am7+} and Proposition \ref{thm:am3} yield 
the above assertion.

\section{Calibrated extremals}

As in the previous section,
we assume throughout this section 
that the critical value $c$ is equal to zero.

\begin{lem}\label{thm:x2}
Let $0<T<\infty$ and $\{(\eta_k,\,v_k,\,l_k)\}_{k\in\N}\subset\SP$.  
Assume that there is a constant $C>0$, independent of $k\in\N$, such that  
\[
\int_0^T \big(L(\eta_k(s),-v_k(s))+g(\eta_k(s))l_k(s)\big)\d s\leq C \ \ \hb{ for }k\in\N.
\] 
Then there exists a triple $(\eta,\,v,\,l)\in\SP$ such that  
\begin{align*}
&\int_0^T\big(L(\eta(s),-v(s))+g(\eta(s))l(s)\big)\d s \\
&\leq \liminf_{k\to\infty}\int_0^T\big(L(\eta_k(s),-v_k(s))+g(\eta_k(s))l_k(s)\big)\d s
\end{align*}
Moreover, for the triple $(\eta,\,v,\,l)$, there is a subsequence $\{(\eta_{k_j},\,v_{k_j},\,l_{k_j})\}$ 
of $\{(\eta_k,\,v_k,\,l_k)\}$ such that as $j\to\infty$,
\begin{align}
& \eta_{k_j}(0)\to \eta(0),\label{eq:x5}\\
& \dot\eta_{k_j}(t)\d t\to \dot\eta(t)\d t \ \ \hb{ weakly-star in }C([0,\,T],\,\R^n)^*,\label{eq:x6}\\
& v_{k_j}(t)\d t\to v(t)\d t \ \ \hb{ weakly-star in }C([0,\,T],\,\R^n)^*,\label{eq:x7}\\
& l_{k_j}(t)\d t \, \to \, l(t)\d t \ \ \hb{ weakly-star in }C([0,\,T])^*. \label{eq:x8}
\end{align}
\end{lem}

Of course, under the hypotheses of the above theorem, the functions 
\[
\eta_{k_j}(t)=\eta_{k_j}(0)+\int_0^t\dot\eta_{k_j}(s)\d s 
\]
converge to $\eta(t)$ uniformly on $[0,\,T]$ as $j\to\infty$. 

\begin{proof} We may assume without loss of generality that 
$\eta_k(t)=\eta_k(T)$, $v_k(t)=0$ and $l_k(t)=0$ for $t\geq T$ and $k\in\N$. 

According to Proposition \ref{thm:s2}, there is a constant $C_0>0$ such that for 
$(\eta,\,v,\,l)\in\SP$, 
\[
|\dot\eta(t)|\vee |l(t)|\leq C_0|v(t)| \ \ \hb{ for a.e. }t\geq 0.  
\] 
Note that for each $A>0$ there is a constant $C_A>0$ such that
\[
L(x,\xi)\geq A|\xi|-C_A \ \ \hb{ for }(x,\xi)\in\ol\gO\tim\R^n. 
\]
From this lower bound of $L$, it is obvious that 
for $(x,\xi,r)\in\ol\gO\tim\R^n\tim\ol\R_+$, if $r \leq C_0|\xi|$, then 
\begin{equation}\label{eq:x2}
L(x,\,\xi)+g(x)r\geq A|\xi|-C_A-C_0\|g\|_\infty|\xi|,  
\end{equation}
which ensures that there is a constant $C_1>0$ such that for   
$(\eta,\,v,\,l)\in\SP$, 
\begin{equation}\label{eq:x3}
L(\eta(s),-v(s))+g(\eta(s))l(s)+C_1\geq 0 \ \ \hb{ for a.e. }s\geq 0.  
\end{equation}

Using (\ref{eq:x3}), we obtain for any measurable $B\subset [0,\,T]$, 
\begin{align*}
&\int_B\big(L(\eta_k(s),-v_k(s))+g(\eta_k(s))l_k(s)+C_1\big)\d s \\ 
&\leq \int_0^T\big(L(\eta_k(s),-v_k(s))+g(\eta_k(s))l_k(s)+C_1\big)\d s\leq C+C_1T. 
\end{align*}
This together with (\ref{eq:x2}), yields 
\begin{equation}\label{eq:x4}
\left(A-C_0\|g\|_\infty\right)\int_B |v_k(s)|\d s
\leq C_A |B|+C+C_1T \ \ \hb{ for }A>0.
\end{equation}
This shows that the sequence $\{|v_k|\}$ is uniformly integrable on $\R_+$.

We choose an increasing sequence $\{k_j\}\subset\N$ so that 
\begin{align*}
&\liminf_{k\to\infty}\int_0^T \big(L(\eta_k(s),-v_k(s))+g(\eta_k(s))l_k(s)\big)\d s\\
&=\lim_{j\to\infty}\int_0^T \big(L(\eta_{k_j}(s),-v_{k_j}(s))
+g(\eta_{k_j}(s))l_{k_j}(s)\big)\d s. 
\end{align*}
Thanks to estimate (\ref{eq:x4}), in view of Proposition \ref{thm:s3}, 
we may assume by replacing $\{k_j\}$ by a subsequence if needed 
that there is a triple $(\eta,\,v,\,l)\in\SP$ such that 
the convergences (\ref{eq:x5})--(\ref{eq:x8}) hold.   
Here we may assume that $(\eta(t),\,v(t),\,l(t))=(\eta(T),\,0,\,0)$ 
for $t\geq T$. 
 
In what follows, we write 
$(\eta_j,\,v_j,\,l_j)$ for $(\eta_{k_j},v_{k_j},l_{k_j})$ 
for notational simplicity. 
It remains to show that 
\begin{align*}
&\int_0^T \big(L(\eta(s),-v(s))
+g(\eta(s))l(s)\big)\d s \\
&\leq 
\lim_{j\to\infty}\int_0^T \big(L(\eta_j(s),-v_j(s))+g(\eta_j(s))l_j(s)\big)\d s. 
\end{align*}
In view of the monotone convergence theorem, we need to show that for each $m\in\N$,
\begin{align*}
&\int_0^T \big(L_m(\eta(s),-v(s))
+g(\eta(s))l(s)\big)\d s \\
&\leq 
\lim_{j\to\infty}\int_0^T \big(L(\eta_j(s),-v_j(s))+g(\eta_j(s))l_j(s)\big)\d s, 
\end{align*}
where 
\[
L_m(x,\,\xi):=\max_{p\in B(0,\,m)}\big(\xi\cdot p-H(x,\,p)\big) \ \ \hb{ for }(x,\,\xi)\in\ol\gO\tim\R^n.
\]
Note that $L_m(x,\,\xi)\leq L_{m+1}(x,\,\xi)$ for $(x,\xi)\in\ol\gO\tim\R^n$ and $m\in\N$, 
$\lim_{m\to\infty}L_m(x,\,\xi)=L(x,\,\xi)$ for $(x,\xi)\in\ol\gO\tim\R^n$ 
and the functions $L_m$ are uniformly continuous on bounded subsets of $\ol\gO\tim\R^n$. 

We fix any $m\in\N$. In view of the selection thorem of Kuratowski and Ryll-Nardzewski, 
we may choose a Borel function $P_m: 
\ol\gO\tim\R^n\to B(0,\,m)$, so that 
\begin{equation}\label{eq:x9}
L_m(x,\,\xi)=\xi\cdot P_m(x,\,\xi)-H(x,\,P_m(x,\,\xi))
 \ \ \hb{ for }(x,\xi)\in\ol\gO\tim\R^n.
\end{equation}
Indeed, if we define the multifunction: $\ol\gO\tim\R^n\to 2^{\R^n}$ 
by 
\[
F(x,\,\xi)=\{p\in B(0,\,m)\mid L_m(x,\,\xi)=\xi\cdot p-H(x,p)\},            
\]
then (i) $F(x,\xi)$ is a nonempty closed set for every $(x,\,\xi)\in\ol\gO\tim \R^n$ 
and (ii) $F^{-1}(K)$ is a closed set whenever $K\subset\R^n$ is closed. 
From (ii), we see easily that $F^{-1}(U)$ is a $F_\gs$-set (and hence a Borel set) 
whenever $U\subset\R^n$ is open. Hence, as claimed above, 
by the thorem of Kuratowski and Ryll-Nardzewski (see, for instance, \cite[Theorem 1.5]{JayneRogers_02}), there exists a function: 
$P_m:\ol\gO\tim\R^n\to \R^n$ such that $P_m(x,\,\xi)\in F(x,\,\xi)$ for all 
$(x,\,\xi)\in \ol\gO\tim\R^n$. 

Set 
\[
p(t)=P_m(\eta(t),\,-v(t)) \ \ \hb{ for }t\geq 0.
\]
Let $\rho_\gep$, with $\gep>0$, be a mollification kernel on $\R$ whose support is contained  
in $[-\gep,\,0]$ and set $p_\gep(t)=\rho_\gep* p(t)$ for $t\geq 0$.   

We fix any $\gep>0$, and observe by the definition of $L$ that
\begin{align*}
I:=&\int_0^T\big(L(\eta_j(s),-v_j(s))+g(\eta_j(s))l_j(s)\big)\d s\\
&\geq \int_0^T\big(-v_j(s)\cdot p_\gep(s) -H(\eta_j(s),\,p_\gep(s))
+ g(\eta_j(s))l_j(s)\big)\d s. 
\end{align*}
From this, in view of (\ref{eq:x5})--(\ref{eq:x8}), we find that
\begin{equation}
\label{eq:x10}\\ 
I\geq \int_0^T\big(-v(s)\cdot p_\gep(s) -H(\eta(s),\,p_\gep(s))
+ g(\eta(s))l(s)\big)\d s.
\end{equation}
Note here that $|p_\gep(s)|\leq m$ for $s\geq 0$ and $p_\gep\to p$ in $L^1(0,\,T)$ as 
$\gep\to 0$. In particular, for some sequence $\gep_k\to +0$, we have 
$p_{\gep_k}(t) \to p(t)$ for a.e. $t\in[0,\,T]$ as $k\to\infty$. 
Sending $\gep
\to 0$ along the sequence $\gep=\gep_k$ and using (\ref{eq:x9}), from (\ref{eq:x10}) we obtain 
\begin{align*}
I\geq&\, \int_0^T \big(-v(s)\cdot p(s)-H(\eta(s)+g(\eta(s))l(s)\big)\d s\\
\geq&\, \int_0^T \big(L_m(\eta(s),\,-v(s))+g(\eta(s))l(s)\big)\d s,
\end{align*}
which completes the proof. 
\end{proof}

\begin{thm}\label{thm:x3} Let $u_0\in C(\ol\gO)$ 
and let $u\in C(\ol\gO\tim\ol\R_+)$ be the unique solution of 
{\em(\ref{eq:i3})--(\ref{eq:i5})}, with $a:=0$. Let $(x,t)\in\ol\gO\tim\R_+$. 
Then there exists 
a triple $(\eta,\,v,\,l)\in\SP(x)$ such that 
\[
u(x,t)=\int_0^t \big(L(\eta(s),\,-v(s))+g(\eta(s))l(s)\big)\d s+u_0(\eta(t)). 
\]
If, in addition, $u\in\Lip(\ol\gO\tim (\ga,\,\gb))$, with $0\leq\ga<\gb\leq\infty$, 
then the triple $(\eta,\,v,\,l)$, restricted to $(\ga,\,\gb)$, belongs to 
$\Lip(\ga,\,\gb)\tim L^\infty(\ga,\,\gb)\tim L^\infty(\ga,\,\gb)$.  
\end{thm}

Here we should note that the infimum on the right hand side of formula (\ref{eq:cn1}) 
is always attained, which is 
a consequence of the above theorem and Theorem \ref{thm:cn1}.

\begin{proof}Fix $(x,t)\in\ol\gO$. By Theorem \ref{thm:cn1}, we can choose a 
sequence $\{(\eta_k,\,v_k,\,l_k)\}\subset\SP(x)$ such that 
\[
u(x,t)=\lim_{k\to\infty}\Big\{\int_0^t\big(L(\eta_k(s),-v_k(s))+g(\eta_k(s))l_k(s)\big)\d s 
+u_0(\eta_k(t))\Big\}. 
\]
By virtue of Lemma \ref{thm:x2}, there are an increasing sequence 
$\{k_j\}\subset\N$ and a $(\eta,\,v,\,l)\in\SP$ such that  
$\eta_{k_j}(s)\to \eta(s)$ uniformly on $[0,\,t]$ and 
\begin{align*}
&\int_0^t\big(L(\eta(s),-v(s))+g(\eta(s))l(s)\big)\d s \\
&\leq \liminf_{k\to\infty}\int_0^t\big(L(\eta_k(s),-v_k(s))+g(\eta_k(s))l_k(s)\big)\d s. 
\end{align*}
Now it is easy to see that 
\[
u(x,t)\geq \int_0^t\big(L(\eta(s),-v(s))+g(\eta(s))l(s)\big)\d s+u_0(\eta(t)),
\]
but we have already the opposite inequality by Theorem \ref{thm:cn1}.   

Now, assume in addition that $u\in\Lip(\ol\gO\tim (\ga,\,\gb))$, where $0\leq\ga<\gb\leq\infty$. 
Let $C>0$ be a Lipschitz 
constant of the fucntion $u$ on the set $\ol\gO\tim[\ga,\,gb]$.  
Let $C_0>0$ be the constant from Propostion \ref{thm:s2}, so that 
$|\dot\eta(s)|\vee l(s)\leq C_0|v(s)|$ for a.e. $s\geq 0$. 
As in the proof of Proposition \ref{thm:s3}, for each $A>0$ 
we choose a constant $C_A>0$ so that 
$L(y,\xi)\geq A|\xi|-C_A$ for $(y,\,\xi)\in\ol\gO\tim\R^n$. 
Fix any finite interval $[a,\,b]\subset (\ga,\,\gb)$. 
Then, with help of the dynamic programming principle,  
we get 
\begin{align*}
&\int_a^b \big(L(\eta(s),-v(s))+g(\eta(s))l(s)\big)\d s  
=u(\eta(b),\,b)-u(\eta(a),\,a)\\
&\leq C(|\eta(b)-\eta(a)|+|b-a|)+C_1(b-a)\leq \int_a^b\big(C|\dot\eta(s)|+C+C_1\big)\d s\\ 
&\,\leq \int_a^b\big(CC_0|v(s)|+C+C_1\big)\d s. 
\end{align*}
On the other hand, for any $A>0$, we have 
\[
\int_a^b \big(L(\eta(s),-v(s))+g(\eta(s))l(s)\big)\d s
\geq \int_a^b\big((A-C_0\|g\|_\infty)|v(s)|-C_A  \big)\d s. 
\]
Combining these, we obtain 
\[
\int_a^b\big((A-C_0\|g\|_\infty-CC_0)|v(s)|-C_A-C-C_1\big)\d s\leq 0. 
\]
We fix $A>0$ so that $A\geq C_0\|g\|_\infty+CC_0+1$ and get 
\[
\int_a^b\big(|v(s)|-C_A-C-C_1\big)\d s\leq 0. 
\]
Since $a,\,b$ are arbitrary as far as $\ga<a<b<\gb$, we conclude from the above that 
$|v(s)|\leq C_A+C+C_1$ for a.e. $s\in(\ga,\,\gb)$. By Proposition \ref{thm:s2}, 
we see that $(\eta,\,v,\,l)\in \Lip(\ga,\,\gb)
\tim L^\infty(\ga,\,\gb)\tim L^\infty(\ga,\,\gb)$.  
\end{proof}

\begin{thm}\label{thm:x4}
Let $\phi\in\Lip(\ol\gO)$ be a solution of 
{\em(\ref{eq:i1}), (\ref{eq:i2})}, with $a:=0$. Let $x\in\ol\gO$. 
Then there is a triple $(\eta,\,v,\,l)\in\SP(x)$ such that for any $t>0$,
\begin{equation}\label{eq:xx}
\phi(x)-\phi(\eta(t))=\int_0^t \big(L(\eta(s),-v(s))+g(\eta(s))l(s)\big)\d s.  
\end{equation}
Moreover, $(\eta,\,v,\,l)\in \Lip(\ol\R_+)\tim L^\infty(\R_+)\tim L^\infty(\R_+)$. 
\end{thm}

Let $\phi$ and $(\eta,\,v,\,l)\in\SP$. 
Following \cite{Fathi-weakKAM}, we call a triple $(\eta,\,v,\,l)\in\SP$ 
{\em calibrated extremal} associated with $\phi$ if (\ref{eq:xx}) holds for all $t>0$. 

\begin{proof}Note that the function $u(x,t)=\phi(x)$ is a solution of 
(\ref{eq:i3}), (\ref{eq:i4}), with $a:=0$. Using Theorem \ref{thm:x3}, we define 
inductively the sequence 
$\{(\eta_k,\,v_k,\,l_k)\}\subset\SP$ as follows. We choose first 
a $(\eta_1,\,v_1,\,l_1)\in\SP(x)$ 
so that
\[
\phi(\eta(0))-\phi(\eta(1))=\int_0^1\big(L(\eta_1(s))+g(\eta_1(s))l_1(s)\big)\d s.
\]
We next assume that $\{(\eta_k,\,v_k,\,l_k)\}_{k\leq j-1}$, with $j\geq 2$, and choose 
a $(\eta_j,\,v_j,\,l_j)\in\SP(\eta_{j-1}(1))$ so that
\[
\phi(\eta_j(1))-\phi(\eta_j(0))=\int_0^1\big(L(\eta_j(s),-v_j(s))
+g(\eta_j(s))l_j(s)\big)\d s. 
\]
Once the sequence $\{(\eta_k,\,v_k,\,l_k)\}_{k\in\N}\subset \SP$ is given, we define the 
$(\eta,\,v,\,l)\in\SP(x)$ by setting 
$(\eta(s+k),\,v_k(s+k),\,l(s+k))=(\eta_k(s),\,v_k(s),\,l_k(s))$ for $k\in\N\cup\{0\}$ 
and $s\in[0,\,1)$. 
It is clear that the triple $(\eta,v,l)$ satisfies (\ref{eq:xx}).  
Thanks to Theorem \ref{thm:x3}, we have $(\eta_k,\,v_k,\,l_k)\in 
\Lip([0,\,1])\tim L^\infty(0,\,1)\tim L^\infty(0,\,1)$ for $k\in\N$. 
Reviewing the proof of Theorem \ref{thm:x3}, we see easily that 
$\sup_{k\in\N}\|v_k\|_{L^\infty(0,\,1)}<\infty$, from which we conclude that 
$(\eta,\,v,\,l)\in \Lip(\ol\R_+)\tim L^\infty(\R_+)\tim L^\infty(\R_+)$.   
\end{proof}

\begin{thm}\label{thm:x5}Let $\phi\in\Lip(\ol\gO)$ be 
a solution of {\em(\ref{eq:i1}), (\ref{eq:i2})}, with $a:=0$ and 
$(\eta,\,v,\,l)\in\SP$ a calibrated extremal associated with $\phi$. Then 
\[
\lim_{t\to\infty}\dist(\eta(t),\,\cA)=0. 
\]
\end{thm}

\begin{proof} According to Lemma \ref{thm:am7}, there are functions 
$\psi\in\Lip(\ol\gO)$ and $f\in C(\ol\gO)$ such that 
$f(x)<0$ for $x\in\ol\gO\setminus\cA$, $f(x)\leq 0$ for $x\in\ol\gO$ and  
$\psi$ is a subsolution of 
(\ref{eq:b1}), with $U=\R^n$.  
Then $u(x,t):=\psi(x)$ is a subsolution of (\ref{eq:i3}), (\ref{eq:i4}), with 
$H$ replaced by $H-f$ and $a:=0$. By comparison, if $w\in C(\ol\gO\tim\ol\R_+)$ is a 
solution of (\ref{eq:i3})--(\ref{eq:i5}), 
with $H$ replaced by $H-f$, $a:=0$ and $u_0:=\psi$, then we get 
$u\leq w$ on $\ol\gO\tim\ol\R_+$. Hence, 
using Theorem \ref{thm:cn1}, with $H$ replaced by $H-f$, we find that for any $t>0$, 
\begin{align}\label{eq:x11}
\psi(\eta(0))\leq&\, 
\int_0^t \big(L(\eta(s),-v(s))+f(\eta(s))+g(\eta(s))l(s)\big)\d s+\psi(\eta(t))\\
=&\,\phi(\eta(0))-\phi(\eta(t))+\psi(\eta(t))+\int_0^t f(\eta(s))\d s.\nonumber
\end{align}
From this we find that 
\[
\inf_{t>0}\int_0^t f(\eta(s))\d s>-\infty \ \ \hb{ or equivalently } \ \  
\int_0^\infty |f(\eta(s))|\d s<\infty, 
\]
which yields
\begin{equation}\label{eq:xx1}
\lim_{t\to\infty}\int_t^{t+1}|f(\eta(s))|\d s=0. 
\end{equation}

Reviewing the proof of Lemma \ref{thm:x2} up to (\ref{eq:x7}), 
since 
\begin{align*}
\int_t^{t+1}\big(L(\eta(s),-v(s))+g(\eta(s)l(s)\big)\d s 
\,&=\phi(\eta(t))-\phi(\eta(t+1)) \\
&\leq 2\|\phi\|_\infty \ \ \hb{ for }t\geq 0,
\end{align*}
we deduce that for any $A>0$ and $t\geq 0$, 
\[
\big(A-C_0\|g\|_\infty\big)\int_t^{t+\gep}
|v(s)|\d s\leq C_A\gep+2\|\phi\|_\infty+C_1, 
\] 
where the constants $C_0,\, C_1,\, C_A$ 
are selected as in the proof of Lemma \ref{thm:x2}. This estimate 
together with Proposition \ref{thm:s3} guarantees that 
$\eta$ is uniformly continuous on $\ol\R_+$. Now, (\ref{eq:xx1}) ensures that 
$\lim_{t\to\infty}f(\eta(t))=0$ and hence $\lim_{t\to\infty}\dist(\eta(t),\cA)=0$. 
\end{proof}

Let $\SP_{-\infty}$ denote the set of all triples 
$(\eta,\,v,\,l):\R\to \ol\gO\tim\R^n\tim\ol\R_+$ such that 
for every $T\geq 0$, the triple $(\eta_T,\,v_T,\,l_T)$ defined on $\ol\R_+$ by 
$(\eta_T(t),\,v_T(t),\,l_T(t))=(\eta(t-T),\,v(t-T),\,l(t-T))$ belongs to 
$\SP$. 

\begin{thm}\label{thm:x6}For any $y\in\cA$ there exists a triple $(\eta,\,v,\,l)\in\SP_{-\infty}$ 
such that $\eta(0)=y$, $\eta(t)\in\cA$ for $t\in\R$ and for any $-\infty<\gs<\tau<\infty$,
\[
\int_\gs^\tau\big(L(\eta(s),-v(s))+g(\eta(s))l(s)\big)\d s=d(\eta(\gs),\eta(\tau)), 
\]
where $d$ is the function on $\ol\gO\tim\ol\gO$ given by {\em(\ref{eq:def-d})}. 
\end{thm}

\begin{proof}Fix $y\in\cA$. By Theorem \ref{thm:am4}, for any $k\in\N$ there is a 
triple $(\bar\eta_k,\,\bar v_k,\,\bar l_k)\in\SP$ such that $\bar\eta_k(0)=\bar\eta_k(\tau_k)=y$ for some $\tau_k>k$ and 
\begin{equation}\label{eq:x12}
\fr 1k>\int_0^{\tau_k}\big(L(\bar\eta_k(s),-\bar v_k(s))+g(\bar\eta_k(s))\bar l_k(s)\big)\d s. 
\end{equation}
For $k\in\N$ we set 
\[
(\eta_k(t),\,v_k(t),\,l_k(t))
=\left\{
\begin{aligned}
&(\bar\eta_k(t),\,\bar v_k(t),\,\bar l_k(t)) && \hb{ for }t\in[0,\,\tau_k],\\
&(\bar\eta_k(t+\tau_k),\,\bar v_k(t+\tau_k),\,\bar l_k(t+\tau_k))\quad && \hb{ for }t\in[-\tau_k,\,0]. 
\end{aligned}
\right.
\]
In view of Proposition \ref{thm:am3}, using (\ref{eq:x12}), we see that 
if $-\tau_k\leq \gs\leq\tau\leq \tau_k$, then 
\begin{align*}
d(y,\eta_k(\gs))\leq&\, \int_{-\tau_k}^\gs\big(L(\eta_k(s),-v_k(s))+g(\eta_k(s))l_k(s)\big)\d s,\\
d(\eta_k(\gs),\eta_k(\tau))\leq&\, \int_\gs^{\tau}\big(L(\eta_k(s),-v_k(s))+g(\eta_k(s))l_k(s)\big)\d s,\\
d(\eta_k(\tau),y)\leq &\,\int_\tau^{\tau_k}\big(L(\eta_k(s),-v_k(s))+g(\eta_k(s))l_k(s)\big)\d s,\\ 
\Big(\int_{-\tau_k}^{\gs}+\int_\gs^\tau+\int_\tau^{\tau_k} \Big)&
\big(L(\eta_k(s),-v_k(s))+g(\eta_k(s))l_k(s)\big)\d s\\
<&\,\fr 2k=\fr 2k +d(y,\,y)\\
\leq&\, \fr 2k+d(y,\,\eta_k(\gs))+d(\eta_k(\gs),\,\eta_k(\tau))
+d(\eta_k(\gs),\,y).
\end{align*}
Consequently we get for $-\tau_k<\gs<\tau<\tau_k$  
\begin{gather*}
\begin{aligned}
d(\eta_k(\gs),\eta_k(\tau))\,&\leq
\int_\gs^{\tau}\big(L(\eta_k(s),-v_k(s))+g(\eta_k(s))l_k(s)\big)\d s\\ 
&<d(\eta_k(\gs),\eta_k(\tau))+\fr 2k,
\end{aligned}
\\
0\leq d(y,\,\eta_k(\tau))+d(\eta_k(\tau),\,y)<\fr 2k. 
\end{gather*}
Hence, applying Lemma \ref{thm:x2}, we find a triple $(\eta,\,v,\,l)\in\SP_{-\infty}$ such that 
$\eta(0)=y$ and for any $-\infty<\gs<\tau<\infty$,
\begin{align}
d(y,\,\eta(\tau))+d(\eta(\tau),\,y)=&\,0,\label{eq:x13}\\
\int_{\gs}^\tau \big(L(\eta(s),\,-v(s))+g(\eta(s)l(s)\big)\d s\leq&\, d(\eta(\gs),\,\eta(\tau)).
\nonumber
\end{align}
The last inequality yields for any $-\infty<\gs<\tau<\infty$, 
\[
d(\eta(\gs),\,\eta(\tau))=\int_\gs^\tau\big(L(\eta(s),\,-v(s))+g(\eta(s))l(s)\big)\d s. 
\]
Theorem \ref{thm:am4} and (\ref{eq:x13}) together guarantee that 
$\eta(t)\in\cA$ for all $t\in\R$. The proof is complete. 
\end{proof}

\bibliographystyle{amsalpha}
\bibliography{ishiiref09}

\bye